\documentclass[]{article}

\usepackage[utf8]{inputenc}
\usepackage[T1]{fontenc}
\usepackage[english]{babel}
\usepackage{setspace}
\usepackage{pstricks}
\usepackage{pst-node}
\usepackage{amsmath}
\usepackage{amssymb}
\usepackage{graphicx}
\usepackage{fullpage}
\usepackage[hypertexnames=false, pdftex]{hyperref}
\hypersetup{ colorlinks = true, linkcolor = black, urlcolor = blue, citecolor = blue }
\usepackage{tikz}
\usepackage{mathtools}
\usepackage[text={15cm,24cm}]{geometry}
\usepackage{enumerate}
\usepackage[justification=centering]{caption}
\usepackage{subcaption}
\usepackage[linesnumbered,ruled]{algorithm2e}
\usepackage{empheq}
\usepackage{etoolbox}
\usepackage{tikz}
\usepackage{array,colortbl,xcolor}
\usepackage{makecell}
\usetikzlibrary{calc,shadings}
\usepackage{pgfplots}
\usepackage{verbatim}
\usetikzlibrary{arrows,shapes}

\newenvironment{proof}{{\bf Proof}:\ }%
{~\ \hfill $\Box$\vspace{0,5cm}}

{~\ \hfill$\Diamond$\vspace{0,5cm}}
\numberwithin{equation}{section}
\newcolumntype{?}{!{\vrule width 2pt}}
\newtheorem{prop}{Property}[section]
\newtheorem{theorem}{Theorem}[section]
\newtheorem{rmk}{Remark}[section]

\newtheorem{coro}[theorem]{Corollary}

\newenvironment{subequations*}{%
	\setcounter{parentequation}{\value{equation}}%
	\setcounter{equation}{0}%
	\ignorespaces
}{%
\setcounter{equation}{\value{parentequation}}%
\ignorespacesafterend
}

\def\addlegendimage{\csname pgfplots@addlegendimage\endcsname}
\title{Robust capacitated trees and networks with uniform demands \footnote{This work was partially supported by the PGMO Programme Gaspard Monge pour l'optimisation et la recherche opérationnelle de la Fondation Mathématique Jacques Hadamard.}}

\author{Cédric Bentz $^{(1)}$, Marie-Christine Costa $^{(2)}$, Pierre-Louis Poirion $^{(3)}$, Thomas Ridremont $^{(4)}$ \\
	{\scriptsize $^{(1)}$ CEDRIC, CNAM, 292 rue Saint-Martin 75003, Paris, France}\\
	{\scriptsize $^{(2)}$ ENSTA ParisTech (and CEDRIC-CNAM) 828, Boulevard des Maréchaux 91762 Palaiseau Cedex}\\
	{\scriptsize $^{(3)}$ Huawei (and CEDRIC-CNAM),
		France}\\
	{\scriptsize $^{(4)}$ CEDRIC, CNAM (and ENSTA ParisTech), 292 rue Saint-Martin 75003, Paris, France}}

\begin{document}
	\pgfdeclarelayer{background}
	\pgfsetlayers{background,main}
\parindent=0cm
\maketitle

\begin{abstract}
We are interested in the design of robust (or resilient) capacitated rooted Steiner networks in case of terminals with uniform demands. Formally, we are given a graph, capacity and cost functions on the edges, a root, a subset of nodes called \emph{terminals}, and a bound $k$ on the number of edge failures. We first study the problem where $k=1$ and the network that we want to design must be a tree covering the root and the terminals: we give complexity results and propose models to optimize both the cost of the tree and the number of terminals disconnected from the root in the worst case of an edge failure, while respecting the capacity constraints on the edges. Second, we consider the problem of computing a minimum-cost survivable network, i.e., a network that covers the root and terminals even after the removal of any $k$ edges, while still respecting the capacity constraints on the edges. We also consider the possibility of protecting a given number of edges. We propose three different formulations: a cut-set based formulation, a flow based one, and a bilevel one (with an attacker and a defender). We propose algorithms to solve each formulation and compare their efficiency.
\end{abstract}

\section{Introduction}
\label{sec:intro}

Nowadays, the design of networks is crucial in many fields such as transport, telecommunications or energy. Here, we are interested in the design of robust (or resilient) networks, for certain notions of robustness which will be described later. Formally, we are given a graph $G$, capacity and cost functions on the edges, a root, and a subset of nodes called \emph{terminals}, and we want to select nodes and edges of $G$ to build a minimum-cost network linking the root to the terminals, while respecting the capacity constraints on the edges. We assume that some edges can break down, and that there is a bound $k$ on the number of edge failures.  This paper deals with the special case where the demand is identical for each terminal, i.e., the flow from the root to each terminal is a constant, and hence can be set to 1 without loss of generality. The uncertainty considered here concerns the breakdowns, and we aim to protect the network to be built against the worst case.\\

We first study the case where the network we want to build is an arborescence (rooted tree). The problem corresponds to the capacitated Steiner tree (or arborescence) problem which has been studied  for instance in \cite{Bentz2016EdgeCapacitated, du2013advances,goemans1993catalog,hwang1992steiner}. We assume that one arc can break down, and we aim at generating a robust arborescence, i.e., an arborescence that minimizes the number of terminals disconnected from the root in the worst case of a breakdown. This can model, in particular, the problem of wiring networks in windfarms, in order to route the energy produced by the wind turbines to the sub-station, while respecting some technical constraints (such as cable capacities, non-splitting constraints, etc.; see \cite{hertz2012optimizing,pillai2015offshore}).

Then, we study the so-called Capacitated Rooted $k$-Edge Connected Steiner Network problem 
:  we aim to design a minimum-cost network in which, after the failure of any $k$ arcs, we can still route one unit of flow from the root to each terminal. This problem is similar to the Survivable Network Design problem 
(see for instance \cite{grotschel1995design}). However, on the one hand, the authors of \cite{grotschel1995design} 
do not take into account the arc-capacities. On the other hand, 
in the latter problem there is a requirement $r_{st}$ for each pair of vertices $(s,t)$; this means that there must be a path from vertex $s$ to $t$ after any $(r_{st}-1)$ arc deletions (in our problem, the edge-survivability requirements are either ($k+1$) or 0). A survey and other work on this problem are available in \cite{goemans1993survivable,kerivin2005design}. In \cite{botton2013benders}, a method based on Benders decomposition is proposed for a problem with hop-constraints.  In \cite{bienstock2000strong,kerivin2002design,rajan2004directed}, the authors take capacities into account, but they allocate it whereas, in our problem, the capacities are fixed. Studies on multicommodity versions of the problem are also available in \cite{dahl1998cutting, stoer1994polyhedral}. Polyhedral studies have also been conducted on problems corresponding to the uncapacitated \cite{baiou1997steiner} or unrooted \cite{biha2000steiner} version of our problem. Eventually, we introduce the Capacitated Protected Rooted $k$-Edge Connected Steiner Network Problem, 
where a subset of arcs may be protected and thus cannot break down.\\

We denote the given underlying graph by $G=(V,E,c,u)$, where $c$ and $u$ are respectively the cost and the capacity functions on the set of edges $E$. We are also given a set $T \subseteq V$ of terminals and a root $r \in V \setminus T$.  The given graph can be directed or undirected (this will be specified in the following), but the Steiner tree or network to be built is always directed from the root towards the terminals. Since we consider a uniform demand at the terminals, the capacity $u_{ij}$, defined as the maximum amount of flow that can be routed from the root to the terminals through the arc $(i,j)$, can also be defined as the maximum number of terminals that can be connected to the root through $(i,j)$. Hence, we can assume without loss of generality that $u(e)$ is a positive integer for each $e \in E$. Given a digraph $G=(V,A)$, we will refer to $\Gamma_{G}^+(v)$ and $\Gamma_{G}^-(v)$ as the set of successors and predecessors of a vertex $v \in V$ in $G$, respectively. In the case where $G$ is an undirected graph, we will refer to $\Gamma_{G}(v)$ as the neighbors of $v$ in $G$.\\

In Section \ref{sec:CRStA}, we study the problem of finding a Steiner or spanning arborescence taking into account both the cost and the number of terminals disconnected from the root in the worst case of an edge breakdown. We provide a complexity result, that proves that deciding whether there exists a spanning arborescence respecting the capacity constraints is an NP-Complete problem; that corresponds to the special case where there is a demand equal to 1 at each node (except the root). We also propose different formulations, considering the criteria either as objectives or as constraints with given bounds (on the costs and/or the maximum number of disconnected terminals). Then, we compare these formulations by testing them on real windfarm data.

In Section \ref{sec:CRkECSN}, we study the capacitated rooted $k$-edge connected Steiner network problem, which amounts to searching for a robust network, i.e., a network that, in the worst case of $k$ arcs failure, can still route one unit of flow from the root to each terminal while respecting the capacity constraints. We give two formulations based on cut-sets and flows respectively, as well as a third one, which is actually a bilevel program whose second level is a min-max problem, with an attacker and a defender. Then, we consider the case where a set of arcs can be protected (and thus cannot break down), and show how to adapt these three formulations in this case. We also propose methods based both on integer linear programming and constraints generation, as well as valid inequalities, to solve each of the formulations we obtained.

Finally, in Section \ref{sec:results}, we compare the efficiency of the methods proposed in the previous section, by testing them on a large set of randomly generated data, before concluding.

\section{Robust arborescences}
\label{sec:CRStA}

In this section, we focus on finding a robust Steiner or spanning arborescence covering the root and the terminals of $G$. Here, the robustness consists in finding a solution which minimizes the number of terminals disconnected from the root in the worst case of an arc failure.

This setting arises in some windfarm cabling problems (see Section \ref{sec:intro}), when technical constraints impose that all electrical flows arriving at any device except the substation must leave it through one and only one cable: an inclusion-wise minimal sub-network of $G$ respecting those constraints then corresponds to a Steiner anti-arborescence. The wind turbines are identical, and the wind is assumed to blow uniformly, so we can assume that each turbine produces one unit of energy. Then, $A$ is the set of all possible cable locations, $r$ is the sub-station collecting the energy and delivering it to the electric distribution network, $T$ represents the set of nodes where a windturbine lies and $V \setminus (\{r\} \cup T)$ is  the set of Steiner nodes, corresponding to possible junction nodes between cables. In that case, the flow is routed from the vertices of $T$ to $r$, and we search for an anti-arborescence. However, the problem is easily seen to be equivalent to the Steiner arborescence problem, by reversing the flow circulation in the solution.

We begin by defining the problem and giving some complexity results, and then we propose mathematical formulations which are tested on real windfarm instances.

\subsection{Definition of problems and complexity results.}

We assume in this section that the graph $G=(V,E)$ is undirected and, when considering a subgraph $G'=(V',A')$ of $G$ to which we give an orientation, we write $V' \subseteq V$ and $A' \subseteq E$ (arcs of $G'$ correspond to edges of $G$).  We define the robust problem without capacity constraint as follows:\\

 \textbf{Robust Steiner Arborescence problem (RStA)}

 \textit{INSTANCE: } A connected graph $G=(V,E,r,T)$ with  $r \in V$ and $T \subseteq V \setminus \{r\}$.

 \textit{PROBLEM: } Find an arborescence $S = (V_S,A_S)$ such that $V_S \subseteq V$, $A_S \subseteq E$ and $T \subset V_S$, which is rooted at $r$ and minimizes the number of terminals disconnected from $r$ when an arc $a$ is removed from $ A_S$, in the worst case.\\

 We also consider the spanning version of the problem (i.e., $T = V \setminus \{r\}$). In this case, the problem is to minimize the number of vertices in the largest (regarding the number of vertices) subarborescence not containing $r$. We define it as follows:\\

 \textbf{Robust Spanning Arborescence problem (RSpA)}

 \textit{INSTANCE: }  A connected graph $G=(V,E,r)$ with $r \in V$.

 \textit{PROBLEM: } Find a spanning arborescence $S$ of $G$, rooted at $r$, which minimizes the size of the largest subarborescence of $S$ not containing $r$.\\

  Obviously,  the largest subarborescence not containing $r$ is rooted at a vertex $v \in \Gamma_G(r)$, and the worst case is the failure of an arc incident to the root. We have the following property:

  \begin{prop}\label{prop:RSpTGammaR}
  	
  	a) There is an optimal solution $S^*=(V,A^*)$ of RSpA  containing $(r,v)$ for all $v \in \Gamma_{G}(r)$ $(\Gamma_{G}(r) =  \Gamma^+_{S^*}(r))$.
  	
  	b) There is an optimal solution $S^*=(V^*,A^*)$ of RStA containing $(r,v)$  for all $v \in V^*_S \cap \Gamma_{G}(r)$.
  \end{prop}

  \begin{proof}
  	Let $S=(V,A_S)$ be an optimal solution  of RSpA such that there is $v \in \Gamma_{G}(r)$ with $(r,v) \notin A_S$, and let $w$ be the predecessor of $v$ in the path from $r$ to $v$ in $S$. If we remove $(w,v)$ from $A_S$ and add $(r,v)$, we obtain a new spanning arborescence  at least as good as $S$, since we have replaced a subarborescence by two subarborescences of smaller sizes. Doing so for each $v \in \Gamma_{G}(r)$ with $(r,v) \notin A_S$ yields a solution $S^*$ verifying the property.
  	
  	The proof is similar for RStA, by replacing $\Gamma_{G}(r)$ by $V^*_S \cap \Gamma_{G}(r)$: if we remove $(w,v)$ from $A_S$ and add $(r,v)$, we obtain a new Steiner arborescence  at least as good as $S$, since we have replaced a subarborescence by two subarborescences spanning at most the same number of terminals.
  \end{proof}

  Notice that the property does not hold if we have capacity constraints, because the capacity of $(r,v)$ can be smaller than the one of $(w,v)$ in the proof above. Let us now introduce the feasibility problem associated with \textbf{RSpA}:\\

  \textbf{Robust Spanning Arborescence Feasibility problem (RSpAF)}

  \textit{INSTANCE: } A connected graph $G=(V,E,r)$ with $r \in V$ and an integer $\beta$ with $1 \leq \beta \leq |V| - 1$.

  \textit{QUESTION: } Is there a spanning arborescence $S=(V_S,A_S)$ of $G$, rooted at $r$, such that the size of any subarborescence of $S$ not containing $r$ is at most $\beta$?

  \begin{theorem}\label{theorem:ComplexityRSpAF}
  	\textbf{RSpAF} is NP-Complete.
  \end{theorem}

  \begin{proof}
  	We introduce the 3-Partition problem \cite{garey1979computers} in order to transform an instance of this problem into a \textbf{RSpAF} one.\\
  	
  	\textbf{3-Partition problem}
  	
  	\textit{INSTANCE: } A finite set $D$ of $3m$ positive integers $d_{i}$, $i=1,..,3m$, and a positive integer $B$ such that $\sum_{i=1,...,3m}d_i = mB$ and $B/4 < d_i < B/2$ $\forall i=1,...,3m$.
  	
  	\textit{QUESTION: } Can $D$ be partitioned into $m$ disjoints subsets $M_{1}, M_{2}, ...,M_{m}$ of three elements such that  the sum of the numbers in each subset is equal to $B$?\\

  	To obtain an instance of \textbf{RSpAF} 
  	from an instance of 3-Partition, we set $\beta=B+1$ and we construct the following graph $G=(V,E)$: we define a root $r$ and $m$ vertices $v_j$ with an edge $[r,v_j]$ for $j=1,...,m$, each vertex $v_j$ corresponding to a set $M_j$. We add $3m$  vertices $w_i$ and the edges $[v_j,w_i]$ for all $j=1,..,m$ and all $i=1,..,3m$, each vertex $w_i$ corresponding to the element $d_i$ of $D$  (the subgraph induced by the vertices $v_j$ and $w_i$ is complete bipartite). Finally, for each $i=1,..,3m$, we add $d_i - 1$ vertices adjacent to $w_i$ : the subgraph induced by those vertices and the vertices $w_i$ is made of $3m$ stars. See  Figure \ref{fig:3PartitionGraph} for a graph representation of a 3-Partition instance with $m=2$, $B=11$ and $D=\{5,3,4,3,4,3\}$. Notice that $\vert V \vert = 1+m+mB$.
  	\\
  	
  	Solving \textbf{RSpAF} on $G$ with $\beta = B + 1$ amounts to finding an arborescence where the size of the subarborescence rooted at each $v_{j}$ is smaller than or equal to $B + 1$.  If there is a solution to \textbf{RSpAF} on $G$, then, from the proof of  Property \ref{prop:RSpTGammaR}, there is a solution $S$ such that $(r,v_j) \in S$  $\forall j=1,...,m$, and each $w_i$ is connected to exactly one $v_j$, otherwise there is a cycle. Given a vertex $v \in S$, let $S(v)$ be the subarborescence of $S$ rooted at $v$: $\forall j=1,...,m$,  we have $|S(v_j)| \leq B + 1$  and $\sum_{j=1,..,m} |S(v_j)| = |V \setminus \{r\}| = mB + m$. Thus,  $\forall j=1,...,m$, $|S(v_j)| = B +1$ and $S(v_j)$ contains $v_j$ and several vertices $w_i$, each having $d_i-1$ successors in $S$. Finally, the constraints $B/4 < d_i < B/2$ imply that, $\forall j=1,...,m$, $v_{j}$ is connected to exactly 3 vertices $w_{i}$ denoted in the following by $w_{j_1}$, $w_{j_2}$ and $w_{j_3}$, and such that $\vert S(w_{j_1}) \vert + \vert S(w_{j_2}) \vert + \vert S({w_{j_3}}) \vert = |S({v_j})|-1 = B$.\\
  	
  	Then, it is easy to obtain a solution to the 3-Partition instance. For each $j=1,..,m$, we set $M_j=\{\vert S(w_{j_1}) \vert,\vert S(w_{j_2}) \vert,\vert S(w_{j_3}) \vert\} = \{d_{j_1},d_{j_2},d_{j_3}\}$. We have $m$ disjoint sets, each of size $B$, which cover exactly $D$. For the instance given in Figure \ref{fig:3PartitionGraph}, a  solution to 3-Partition can be associated with the arborescence given in thick: $M_{1} = \{5,3,3\}$ and $M_{2} = \{4,4,3\}$.\\
  	
  	  Similarly,  from a solution to the 3-Partition instance, it is easy to obtain a solution $S$  to \textbf{RSpAF} for the associated graph $G$. 
  	  \\
  	
  	The 3-Partition problem is NP-Complete in the strong sense, meaning that it remains NP-Complete even if the integers in $D$ are bounded above by a polynomial in $m$. Thus, the reduction can be done in polynomial time and  \textbf{RSpAF}, which is clearly in NP, is NP-Complete.
  \end{proof}

   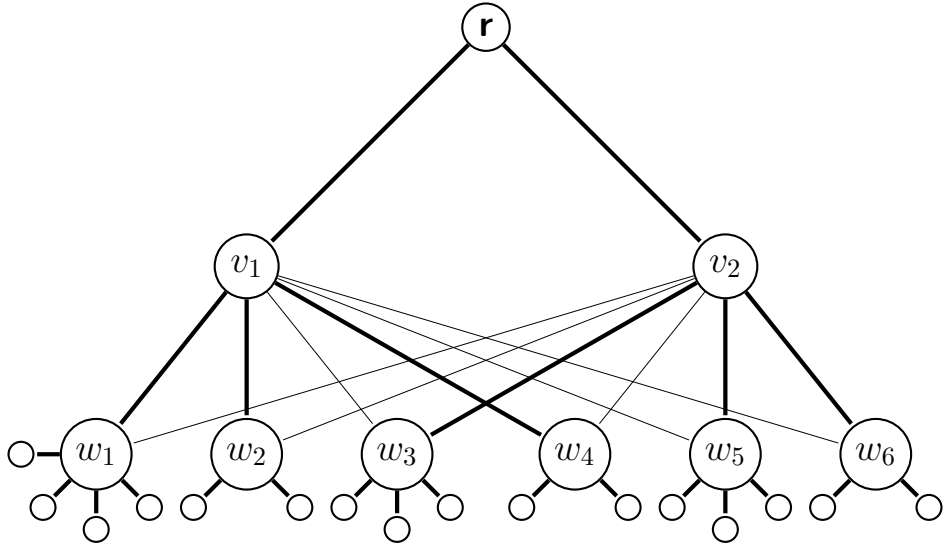
\begin{figure}
   	\centering
   	\begin{tikzpicture}[-,>=stealth,shorten >=1pt,auto,node distance=1.5cm,
   	thick]
   	\tikzstyle{gsm}=[circle,draw,font=\sffamily\Large\bfseries]
   	\node[gsm] (r) {r};
   	\node[gsm,below right of =r, node distance = 4.5cm] (v2) {$v_{2}$};
   	\node[gsm,below left of=r, node distance = 4.5cm] (v1) {$v_{1}$};
   	\node[gsm] (a2) [below of=v1, node distance = 2.5cm]{$w_{2}$};
   	\node[gsm] (a1) [left of=a2, node distance = 2cm]{$w_{1}$};
   	\node[gsm] (a3) [right of=a2, node distance = 2cm]{$w_{3}$};
   	\node[gsm] (a5) [below of=v2, node distance = 2.5cm]{$w_{5}$};
   	\node[gsm] (a4) [left of=a5, node distance = 2cm]{$w_{4}$};
   	\node[gsm] (a6) [right of=a5, node distance = 2cm]{$w_{6}$};
   	\node[gsm] (a1w1) [below left of=a1, node distance = 1cm] {};
   	\node[gsm] (a1w2) [left of=a1, node distance = 1cm] {};
   	\node[gsm] (a1w3) [below of=a1, node distance = 1cm] {};
   	\node[gsm] (a1w4) [below right of=a1, node distance = 1cm] {};
   	\node[gsm] (a2w1) [below left of=a2, node distance = 1cm] {};
   	\node[gsm] (a2w2) [below right of=a2, node distance = 1cm] {};
   	\node[gsm] (a3w1) [below left of=a3, node distance = 1cm] {};
   	\node[gsm] (a3w2) [below of=a3, node distance = 1cm] {};
   	\node[gsm] (a3w3) [below right of=a3, node distance = 1cm] {};
   	\node[gsm] (a4w1) [below left of=a4, node distance = 1cm] {};
   	\node[gsm] (a4w2) [below right of=a4, node distance = 1cm] {};
   	\node[gsm] (a5w1) [below left of=a5, node distance = 1cm] {};
   	\node[gsm] (a5w2) [below of=a5, node distance = 1cm] {};
   	\node[gsm] (a5w3) [below right of=a5, node distance = 1cm] {};
   	\node[gsm] (a6w1) [below left of=a6, node distance = 1cm] {};
   	\node[gsm] (a6w2) [below right of=a6, node distance = 1cm] {};

   	\path[every node/.style={font=\sffamily\small}]
   	(r) edge[ultra thick] node {} (v1)
   	edge[ultra thick] node {} (v2)
   	(v1) edge[ultra thick] node {} (a1)
   	edge[ultra thick] node {} (a2)
   	edge[ultra thin] node {} (a3)
   	edge[ultra thick] node {} (a4)
   	edge[ultra thin] node {} (a5)
   	edge[ultra thin] node {} (a6)
   	(v2) edge[ultra thin] node {} (a1)
   	edge[ultra thin] node {} (a2)
   	edge[ultra thick] node {} (a3)
   	edge[ultra thin] node {} (a4)
   	edge[ultra thick] node {} (a5)
   	edge[ultra thick] node {} (a6)
   	(a1) edge[ultra thick] node {} (a1w1)
   	edge[ultra thick] node {} (a1w2)
   	edge[ultra thick] node {} (a1w3)
   	edge[ultra thick] node {} (a1w4)
   	(a2) edge[ultra thick] node {} (a2w1)
   	edge[ultra thick] node {} (a2w2)
   	(a3) edge[ultra thick] node {} (a3w1)
   	edge[ultra thick] node {} (a3w2)
   	edge[ultra thick] node {} (a3w3)
   	(a4) edge[ultra thick] node {} (a4w1)
   	edge[ultra thick] node {} (a4w2)
   	(a5) edge[ultra thick] node {} (a5w1)
   	edge[ultra thick] node {} (a5w2)
   	edge[ultra thick] node {} (a5w3)
   	(a6) edge[ultra thick] node {} (a6w1)
   	edge[ultra thick] node {} (a6w2);
   	
   	\end{tikzpicture}
   	\centering
   	\captionof{figure}{Graph and \textbf{RSpAF} solution resulting from the 3-Partition instance in which $m=2$, $B=11$, $D=\{5,3,4,3,4,3\}$}
   	\label{fig:3PartitionGraph}
   \end{figure}

  \textbf{RSpAF} being NP-Complete, \textbf{RSpA} is NP-Hard, and so is \textbf{RCStA} because it is a generalization of \textbf{RSpA}. Let us now consider capacity constraints on the edges. 
\textbf{RSpAF} can be seen as a special case of the general capacitated spanning arborescence problem where the demand at each node is an integer (our demands are all equal to 1), and hence from Theorem \ref{theorem:ComplexityRSpAF} we obtain the following corollary:

  \begin{coro}
  	 Given a graph $G=(V,E,r,d,u)$ where $d$ represents the (integral) demands at each node and $u$ the capacities of the edges, the problem of deciding whether there exists a spanning arborescence of $G$, rooted at $r$ and respecting the capacities, is NP-Complete (even if $u$ is a uniform function and all demands are equal to 1).
  \end{coro}

This extends the following result due to Papadimitriou \cite{papadimitriou1978}: given two positive values $C$ and $K$ and a graph $G=(V,E,r,c)$ where $c$ is a cost function on the edges, the problem of deciding whether there exists a spanning arborescence $S$ of $G$ rooted at $r$, such that each subarborescence of $S$ not containing $r$ contains at most $K$ vertices, and with total cost at most $C$, is NP-Complete.\\

The complexity results given in this section concern undirected graphs, and so the more general case of directed graphs too, since an undirected graph can be transformed into a directed one by replacing each edge by two opposite arcs. If we consider problems with capacity constraints, we give the same capacity to both opposite arcs: since we search for an arborescence, only one of them will appear in the solution.\\

 In the following, we study the more general following problem, which is hence also NP-hard:\\

 \textbf{Robust Capacitated Steiner Arborescence problem (RCStA)}

 \textit{INSTANCE: } A connected graph $G=(V,E,r,T,u)$ with  $r \in V$, $T \subseteq V \setminus \{r\}$ and $u$ a positive integer function on $E$.

 \textit{PROBLEM: } Find an arborescence $S = (V_S,A_S)$ with $V_S \subseteq V$ and  $A_S \subseteq E$, rooted at $r$ and spanning the terminals of $T$, which respects the arc capacities and minimizes the number of terminals disconnected from $r$ when an arc $a$ is removed from $ A_S$, in the worst case.\\

 \subsection{Mathematical formulations and tests}

 In this section we propose formulations for robust Steiner problems where the robustness is considered either as a constraint with the objective of minimizing the cost, or as an objective with or without constraints on the cost. Moreover, we study two kinds of robustness by considering  worst or average consequences of breakdowns.\\

 Let $G=(V,A,r,T,u,c)$ be a directed graph  with a root $r$, a set of terminals $T$, and capacity  and cost functions, respectively denoted by $u$  and $c$, on the arcs. As seen before, if $G$ is undirected, then we replace each edge by two opposite arcs with the same capacity and cost. To formulate the different problems, for each arc $(i,j) \in A$ we introduce the 0-1 variable $y_{ij}$ and the integer variable $x_{ij}$, where $y_{ij}$ equals 1 if and only if the arc $(i,j)$ is selected in the final solution, and $x_{ij}$ represents the number of terminals connected to the root through the arc $(i,j)$, or equivalently the number of terminals in the subarborescence rooted at $j$. We introduce the following polyhedron $\mathcal{T}$:

   \[
  \mathcal{T}=\left\{
    {x \in \mathbb{N}^{|A|}, \ y \in \{0,1\}^{|A|}} \left|
  \begin{array}{ll}
   \sum\limits_{(i,j) \in A} x_{ij} \ -  \sum\limits_{(j,k) \in A} x_{jk}\  =  \left\{
  \begin{array}{l l}
  |T| & \ \text{if } j = r\\
  -1 & \ \text{if } j \in T\\
  0 & \ \text{else}\\
  \end{array}
  \right.  & \forall j \in V\\
  \sum\limits_{(i,j) \in A} y_{ij} \quad \leq \quad 1 \quad & \forall j \in V \setminus \{r\}\\
  \ \ x_{ij} \quad \leq \quad u_{ij}y_{ij} \quad & \forall (i,j) \in A
  \end{array}
  \right.
  \right\}
  \]\\

In the following, we write $(x,y) \in \mathcal{T}$ when we consider a couple of variables verifying the contraints of $\mathcal{T}$. The first set of constraints in $\mathcal{T}$ ensures both the conservation of the number of terminals connected through each Steiner vertex $j \in V$ (flow conservation) and the connection of the root to all terminals. The second set of constraints ensures that the solution is an arborescence, i.e., that each vertex has at most one predecessor. Finally, the third set ensures that there is no flow on a non existing arc, and that the number of terminals connected through an arc $(i,j) \in A$ does not exceed its capacity. In the following, the relative gap between two costs will be denoted by $\Delta$. The well-known problem of the Capacitated Steiner Arborescence (\textbf{CStA}) can be formulated as follows \cite{bousba1991finding}:

  \[
  \textbf{CStA} \quad \left| \quad \min\limits_{(x,y) \in \mathcal{T}}  \quad  \sum\limits_{(i,j) \in A} c_{ij}y_{ij} \right.\\
  \]

  As explained previously, we evaluate the robustness by considering the number of terminals disconnected from the root in a worst scenario, that is, the maximum number of terminals connected through an arc incident to the root, which is equal to $ \max_{j \in \Gamma^+_G(r)} \ x_{rj}$. Let $R$ be a fixed bound on this value: we may disconnect at most $R$ terminals from the root by deleting an arc. We propose the following formulation for the Capacitated Steiner Arborescence  with bounded robustness ($\textbf{CStA}_{bounded-robust}$):

   \[
   \textbf{CStA}_{bounded-robust}\quad \left|
   \begin{array}{ll}
   \quad \min\limits_{(x,y) \in \mathcal{T}} & \quad  \sum\limits_{(i,j) \in A} c_{ij}y_{ij}\\
   \quad \text{ s.t. }  & \quad  \max\limits_{j \in \Gamma^+_G(r)} \ x_{rj} \ \leq R
   \end{array}
   \right. \\
   \]

  Let us now consider the robustness as an objective. Note that the default objective function is to minimize the cost of the solution. If a model uses another objective function, then its name will start by a given letter, e.g., $R$ if we want to optimize the worst-case robustness. We propose the following formulation for \textbf{RCStA}:

  \[
  \textbf{RCStA} \quad \left| \quad \min\limits_{(x,y) \in \mathcal{T}}  \quad  \max\limits_{j \in \Gamma^+_G(r)} \ x_{rj} \right.\\
  \]

  Since this formulation does not take the cost into account, we also propose a new formulation where we bound the cost of a solution by a given value $C$:

    \[
  \textbf{RCStA}_{bounded-cost} \quad \left|
  \begin{array}{ll}
  \quad \min\limits_{(x,y) \in \mathcal{T}} & \quad  \max\limits_{j \in \Gamma^+_G(r)} \ x_{rj}\\
  \quad \text{ s.t. }  & \quad\sum\limits_{(i,j) \in A} c_{ij}y_{ij} \ \leq \ C
  \end{array}
  \right. \\
  \]

 However, the previous models only consider the worst-case of a breakdown. It appears that it could also be interesting to "balance" the tree in order to reduce the loss due to an "average breakdown". To this end, we consider arc failures at each vertex and not only at the root, i.e., for each $i \in V$, we consider the worst case of a breakdown of an arc leaving $i$. This corresponds, for each $i \in V$, to the maximum number of terminals that cannot be reached from the root in case of a breakdown of an arc $(i,j)$, $j \in \Gamma^+_G(i)$, or equivalently to the maximum flow on an arc $(i,j)$, $j \in \Gamma^+_G(i)$. We define the "balanced robustness" as the sum of these values: $\sum_{i \in V} \max_{j \in \Gamma^+_G(i)} \ x_{ij}$. 

 We will use the letters $BR$ to refer to models where one wants to optimize the balanced robustness. We propose formulations similar to the previous ones for the Capacitated Steiner Arborescence with bounded balanced robustness, where we bound the balanced robustness of a solution by a given value $BR$:

   \[
   \textbf{CStA}_{bounded-balanced\_robust}\quad \left|
   \begin{array}{ll}
   \quad \min\limits_{(x,y) \in \mathcal{T}} & \quad  \sum\limits_{(i,j) \in A} c_{ij}y_{ij}\\
   \quad \text{ s.t. }  & \quad  \sum\limits_{i \in V} \max\limits_{j \in \Gamma^+_G(i)} \ x_{ij}\ \leq BR
   \end{array}
   \right. \\
   \]

The following formulation aims at computing the best balanced robustness:

 \[
  \textbf{BRCStA} \quad \left| \quad \min\limits_{(x,y) \in \mathcal{T}}  \quad \sum\limits_{i \in V} \max\limits_{j \in \Gamma^+_G(i)} \ x_{ij} \right.\\
  \]

Moreover, we can keep this latter objective while bounding both the worst-case robustness (by $R$) and the cost of the solution  (by $C$). We obtain:

 \[
   \textbf{BRCStA}_{bounded-robust-cost} \quad \left|
  \begin{array}{ll}
  \quad \min\limits_{(x,y) \in \mathcal{T}} & \quad \sum\limits_{i \in V} \max\limits_{j \in \Gamma^+_G(i)} \ x_{ij}\\
  \quad \text{ s.t. }  & \quad  \max\limits_{j \in \Gamma^+_G(r)} \ x_{rj} \ \leq R\\
  \quad   & \quad\sum\limits_{(i,j) \in A} c_{ij}y_{ij} \ \leq \ C
  \end{array}
  \right. \\
  \]

 We tested those formulations on real wind farm data sets. Even if the number of instances is small, the results are interesting to analyze, and we can compare the robustness, costs and structures of the solutions.
  Data parameters and results  are available respectively in Tables \ref{tab:Parameters} and \ref{subtab:RStPWithoutBound}.  Figure \ref{fig:resultsBRStP} allows to visually compare the arborescences obtained according to the different models for the fourth data set (the filled circles correspond to terminals).\\

   Figure \ref{subfig:StTP} gives an optimal (non robust) capacitated Steiner arborescence (optimal solution of \textbf{CStA}); let us denote its cost  by $C^*$. This arborescence cannot be qualified as robust since, in the worst case, all terminals can be disconnected by deleting the only arc incident to the root. Furthermore, the tree has a large depth, and hence the balanced robustness is not good either. This proves the importance of searching for a more robust solution. We consider first the worst case, \textbf{RCStA}, and we denote by $R^*$ the best robustness, i.e., the minimum value of the loss of terminals in the worst case of a single arc deletion. See Figure \ref{subfig:RCStA} for the associated solution on the test instance.
 Then, to obtain the minimum cost of a most robust solution, denoted by $C^*_{R^*}$, we solve $\textbf{CStA}_{bounded-robust}$ with $R=R^*$: notice that the constraint is saturated in any feasible solution. Then, $\Delta_{Crob}= (C_{R^*}^*-C^*)/C^*$   represents the "cost of robustness", i.e., the percentage of augmentation of the cost to get a robust solution.

In the same way, let $BR^*$ be the best balanced robustness (optimal value of \textbf{BRCStA}, not given in the table); see Figure \ref{subfig:BRCStA} for the associated solution on the test instance. The cost of a solution with the best balanced robustness, denoted by $C_{BR^*}^*$, is obtained by solving $\textbf{CStA}_{bounded-balanced\_robust}$ with $BR=BR^*$, and $\Delta_{Cbrob}= (C_{BR^*}^*-C^*)/C^*$   represents the "cost of balanced robustness",  i.e., the percentage of augmentation of the cost of a non robust arborescence to get a balanced robust solution.

We also study the behaviour of the robustness when we bound the cost to a value close to the one of an optimal non robust arborescence : $R_8$ (resp. $R_{12}$) corresponds to the optimal value of $\textbf{RCStA}_{bounded-cost}$ with a bound $C = 1.08 C^*$ (resp. $C = 1.12 C^*$).

We now analyse the results. The cost of robustness is quite variable on those instances (from 9 to 24\%) but remains rather low. On the contrary, we can see that the optimization of the average robustness is way more expensive (raise from 33\% to 64\% of the cost) because  it involves significantly more edges (see  Figure \ref{subfig:BRCStA}).\\

As we can see on Table \ref{subtab:RStPWithoutBound}, a cost augmentation of 8\% or 12\% on the optimal cost  can result in a solution with a good value of worst-case robustness for some instances: instances 2 and 4 present an excellent value of such robustness with only a cost augmentation of 8\%, while instances 1 and 3 have a rather good one with a cost augmentation of 12\%.\\

Finally, we compare the optimal robustness $R^*$ to the robustness of the balanced arborescence $S_b$ obtained by solving $\textbf{BRCStA}$, i.e., we compute in $S_b$ (see Figure \ref{subfig:BRCStA})  the maximum number of terminals which are disconnected after the deletion of an arc incident to the root. Let $R_{BR^*}$ be this number, shown in the last column of Table \ref{subtab:RStPWithoutBound}. For the test instances, the values of $R^*$ and $R_{BR^*}$ are the same, which means that $S_b$ is a good solution for both the worst and balanced robustness, but we have seen before that its cost is high. Indeed, for these instances, we see that forcing a solution with $R=R^*$ to be optimally balanced increases the cost by at least 33 \%. Nevertheless, there is no guarantee in the general case that the best balanced solution also has the best robustness in the worst case, although the arcs incident to the root are involved in the computation of the balanced robustness. \\

   \begin{table}[h!]
  	\begin{subtable}[t]{0.35\linewidth}
  		\centering
  		\begin{tabular}{|r|r|r|r|}
  			\hline
  			Set & |V| & |E| & |T| \\
  			\hline
  			1 & 91    & 220    & 42  \\
  			\hline
  			2 & 143    & 382    & 40  \\
  			\hline
  			3 & 220    & 510    & 88  \\
  			\hline
  			4 & 255    & 662    & 73  \\
  			\hline
  		\end{tabular}
  		\caption{Data parameters}
  		\label{tab:Parameters}
  	\end{subtable}
  	\hfill
    \begin{subtable}[t]{0.65\linewidth}
  		\centering
  		\begin{tabular}{|r|r|r|r|r|r|r|}
  			\hline
  			Set & $R^*$ & $R_{8}$ & $R_{12}$ & $\Delta_{Crob}$ & $\Delta_{Cbrob}$ & $R_{BR^*}$ \\
  			\hline
  			1 & 21 & 35 & 29 & 0.18 & 0.56 & 21 \\
  			\hline
  			2 & 20 & 21 & 20 & 0.09 & 0.64 & 20 \\
  			\hline
  			3 & 22 & 32 & 30 & 0.24 & 0.33 & 22 \\
  			\hline
  			4 & 37 & 41 & 38 & 0.19 & 0.37 & 37 \\
  			\hline
  		\end{tabular}
  		\caption{Results on robust arborescences}
  		\label{subtab:RStPWithoutBound}
  	\end{subtable}
  	\caption{Results on robust arborescences and data parameters}
  	\label{tab:resultsRStP}
  	
  \end{table}

  \begin{figure}
  	\begin{subfigure}{0.35\linewidth}
  		\begin{center}
  			\includegraphics[scale=0.1]{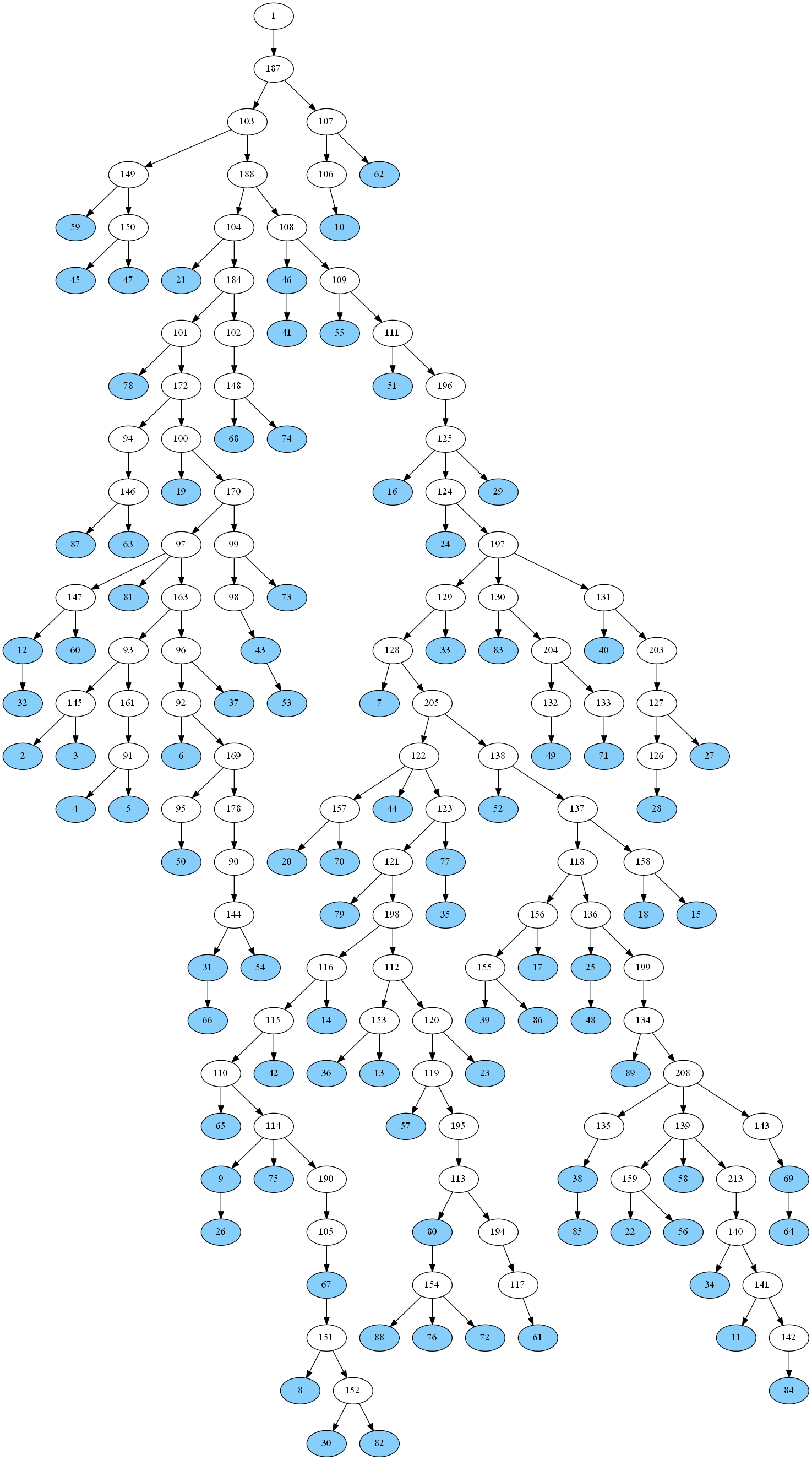}
  			\caption{$CStA$}
  			\label{subfig:StTP}
  		\end{center}
  	\end{subfigure}
  	\hfill
  	\begin{subfigure}{0.64\linewidth}
  		\begin{center}
  			\includegraphics[scale=0.105]{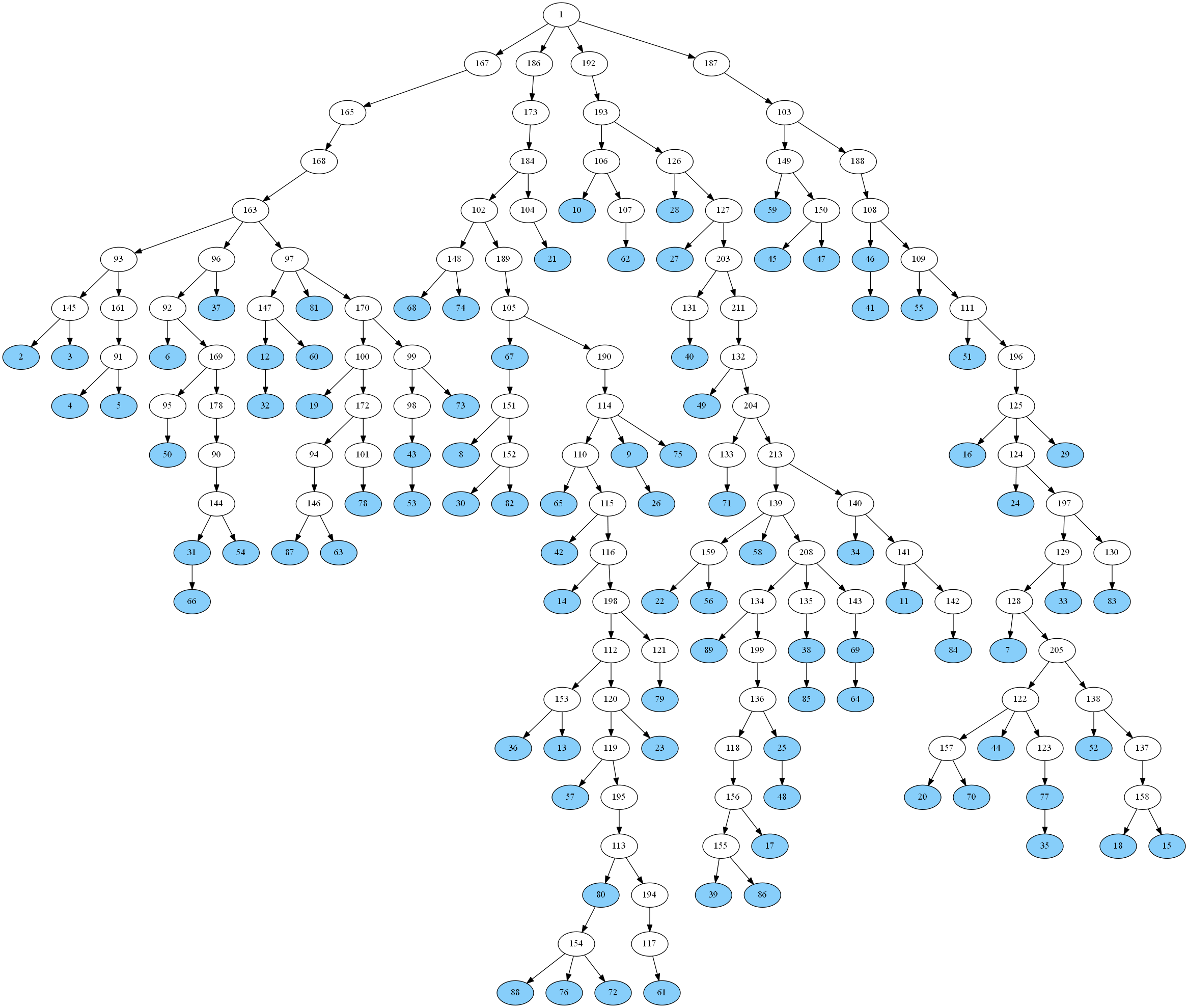}
  			\caption{$RCStA$}
  			\label{subfig:RCStA}
  		\end{center}
  	\end{subfigure}
  	  	\begin{subfigure}{1\linewidth}
  	  		\begin{center}
  	  			\includegraphics[scale=0.09]{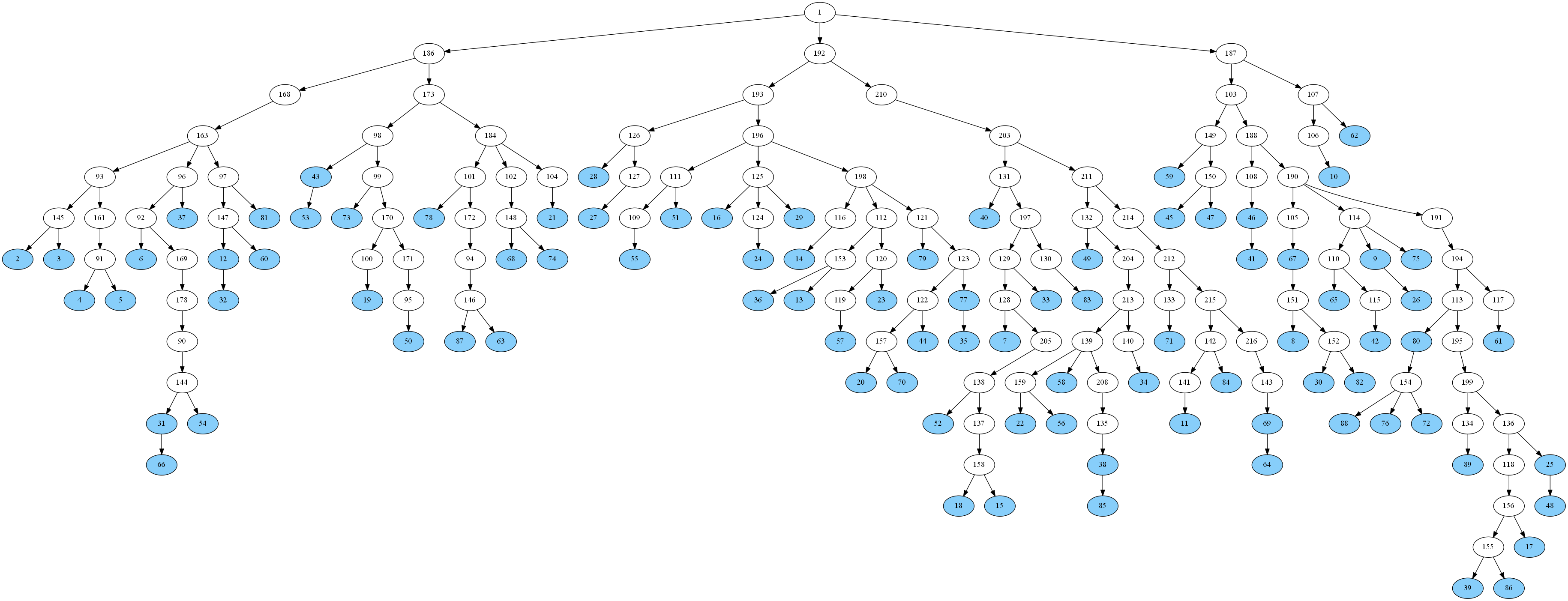}
  	  			\caption{$BRCStA$}
  	  			\label{subfig:BRCStA}
  	  		\end{center}
  	  	\end{subfigure}
  	\begin{subfigure}{1\linewidth}
  		\begin{center}
  			\includegraphics[scale=0.1]{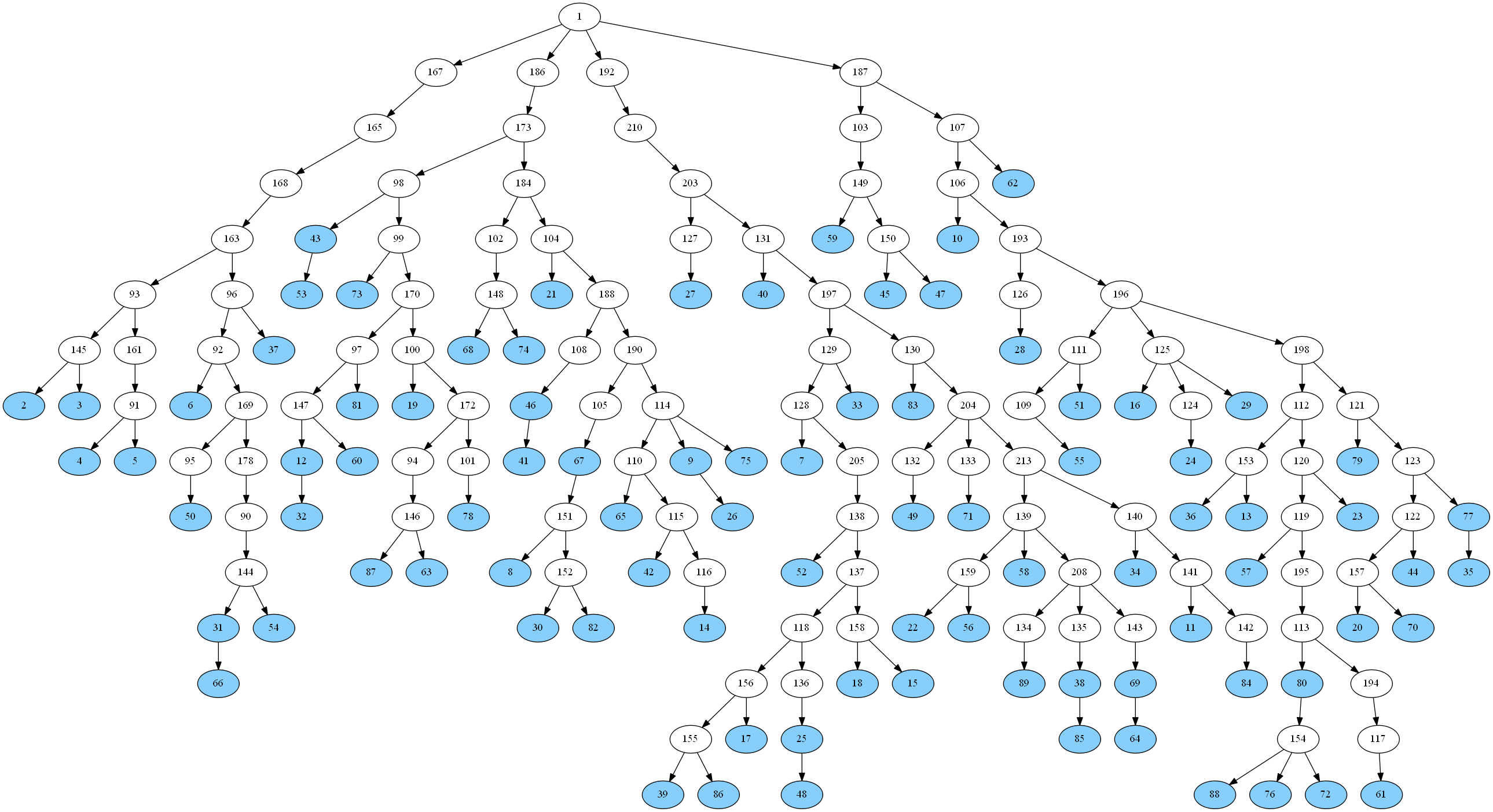}
  			\caption{$BRCStA_{bounded-robust-cost}$}
  			\label{subfig:BRStPWithBounds}
  		\end{center}
  	\end{subfigure}
  	\hfill
  	\caption{Resulting arborescences}
  	\label{fig:resultsBRStP}
  \end{figure}

  When trying to minimize the number of disconnected terminals in the worst
case (see \textbf{RCStA} in Figure \ref{subfig:RCStA}), we have seen that the associated solutions have a reasonable cost, but the average robustness is not good, since the tree remains too deep. When finding the Balanced Steiner arborescence (see \textbf{BRCStA} in Figure \ref{subfig:BRCStA}), the balanced robustness is optimal and the robustness in the worst case is fine, but the cost can be really high (a raise of the optimal cost to 64\% on those data sets). Adding bounds on both cost and worst-case robustness, while minimizing the balanced robustness (see \textbf{BRCStA}$_{bounded-robust-cost}$ in Figure \ref{subfig:BRStPWithBounds}), yields a solution which has both a reasonable cost and a really good worst-case and balanced robustness, and hence it seems that it actually yields the best compromise between the three optimization criteria (the cost and the two types of robustness considered here).\\

\section{Capacitated Rooted k-Edge Connected Steiner Network problem (\textbf{CRkECSN})}
\label{sec:CRkECSN}

\subsection{Definitions and notations}

In this section, we study the problem of designing networks which are resilient to a given number of arc-failures. A feasible solution to the problem we shall consider is then a network rooted at a given root and covering a given set of terminals, and such that, after deleting any $k$ arcs, it is still possible to route a unit of flow from the root to each terminal, while respecting given capacities on the arcs. Formally, we define the following problem:\\

  \textbf{Capacitated Rooted k-Edge Connected Steiner Network problem (CRkECSN)}

  \textit{INSTANCE: } A connected digraph $G=(V,A,r,T,u,c)$ with $r \in V$, $T \subseteq V \setminus \{r\}$, a capacity function $u$ on $A$, a cost function $c$ on $A$, an integer $k$ with $1 \leq k \leq |A| - 1$.

  \textit{QUESTION: } Find a subset $A' \subseteq A$ of minimum cost such that there is a feasible flow (i.e. respecting the arc capacities) routing a unit of flow from $r$ to each vertex of $T$ in the subgraph of $G$ induced by $A'$, even if any $k$ arcs in $A'$ are deleted. \\

  As we consider uniform production at the terminals, we also assume that $u$ is an integer function. \\

  \begin{prop}
  	For $k \in \mathbb{N^*}$, there are at least $k+1$ arc-disjoint paths between the root and each terminal in any feasible solution. Furthermore, any inclusion-wise minimal feasible solution induces at least a 2-edge-connected graph in
  the underlying undirected graph.
  \end{prop}

  \begin{proof}
 The first part of the property is trivial. Let $G'$ be an inclusion-wise minimal feasible solution and let assume $G'$ is not 2-edge-connected in the underlying undirected graph. Then there exists at least one edge whose removal cuts $G'$ into two parts. If the part that does not include the root contains terminals, then $G'$ is clearly not a feasible solution because, if we remove this edge, then at least one terminal cannot be reached from the root. Otherwise, $G'$ is not inclusion-wise minimal because, if we remove this edge, then the resulting graph is still a feasible solution. Hence, any inclusion-wise minimal feasible solution induces at least a 2-edge-connected graph.
  \end{proof}

  In order to simplify the formulations proposed in the next sections, we add to the input graph a vertex $s$ (which corresponds to a fictive sink) connected to every terminal $t \in T$ by a fictive arc $(t,s)$ with $c_{ts} = 0$ and $u_{ts} = 1$. Then, $s$ is added to $V$ and the fictive arcs are added to $A$, and we denote by $A_I$ the set of initial arcs. Finding a flow which routes one unit of flow between $r$ and each terminal in the input graph is equivalent to finding a flow of value $|T|$ from $r$ to $s$ in the transformed graph. For each subset $S \subset V$, let $\delta^-(S)$ be the set of arcs entering $S$. We have $\delta^-(S)=\{(i,j) \in A \ | \ i \in V \setminus S, \ j \in S\}$.

  \subsection{Formulations}\label{subsec:formulations}

  \subsubsection{Cutset formulation} \label{subsubsec:cut-set}

  We introduce, for each $(i,j) \in A$, a binary variable $y_{ij}$ equal to 1 if the arc $(i,j)$ is selected in $A'$, 0 otherwise. Consider the $r-s$ cuts $[V \setminus V_S,V_S] $ with $ V_S \subset V $, $r \in V \setminus V_S$, $s \in V_S$ and $V_S \neq \{s\}$, and  let $\mathcal{S}$ be the set of all the associated cut-sets $S$ in $A$, i.e. $S = \delta^{-}(V_S)$ for each $V_S$. $\mathcal{S}$ is the set of $r-s$ cutsets except the one implying only fictive arcs. Notice that if $S \in \mathcal{S}$ then $S \cap A'$ is a cut-set in the selected network.
  For any set $S \in \mathcal{S}$, let $C_k^S$ be the set of subsets of $S$ of size  $k$.
  For each $S$, we define $M_{S}$ as the maximum capacity of a subset of $k$ selected arcs of $S$:

  \begin{subequations}
  	\begin{empheq}{align}
  	M_{S} \quad = \quad \max\limits_{C \in C^S_k} \sum\limits_{(i,j) \in C}  u_{ij}y_{ij} \label{eq:MS}
  	\end{empheq}
  \end{subequations}

  $M_S$ corresponds to the maximum capacity that can be lost in the cut-set $S$ after the deletion of $k$ arcs. We propose the following cutset formulation:

  \begin{subequations}
  	\begin{empheq}{align}
  	\min\limits_{y}  & \quad  \sum\limits_{(i,j) \in A} c_{ij}y_{ij} \quad \notag \\
  	\text{s.t.}  &  \sum\limits_{(i,j) \in S}u_{ij}y_{ij} \ - \ M_{S} \ \geq \ |T| & \quad \forall S \in \mathcal{S} \quad \label{cstr:cutset-cut}\\
  	&  \ \quad y_{ij} \in \{0,1\} \qquad & \forall (i,j) \in A \quad \notag
  	\end{empheq}
  \end{subequations}

  The constraints (\ref{cstr:cutset-cut}) ensure that, for each cut, the capacity of the cut after the worst-case deletion of $k$ arcs of the cut-set is at least equal to the number of terminals, i.e., one can still route $|T|$ units of flow from $r$ to $s$ while respecting the capacities. They are necessary to every feasible solution. Indeed, they ensure that, for each cut-set in the graph induced by the arcs $(i,j)$ such that $y_{ij}=1$, the capacity of the cutset minus the $k$ maximal arc capacities of the cut is greater than or equal to the number of terminals. If a constraint is not satisfied, it means that, by removing $k$ arcs, the capacity of the min-cut in the graph induced by $y$ becomes smaller than $|T|$. Those constraints are also sufficient to ensure a feasible network. Indeed, if they are satisfied for each cut-set $S \in \mathcal{S}$, it means that you cannot find a set of $k$ arcs whose removal will induce a min-cut with capacity smaller than $|T|$ (which is a necessary and sufficient condition for the existence of a flow of value $|T|$).\\

 The constraints (\ref{cstr:cutset-cut}) are non linear because of the use of the maximum operator in the definition of $M_{S}$. To linearize it, we can rewrite (\ref{cstr:cutset-cut}) as follows:

  \begin{subequations}
  	\begin{empheq}{align}
  	\sum\limits_{(i,j) \in S \setminus C} u_{ij}y_{ij} \ \quad \geq \quad |T| \qquad \forall S \in \mathcal{S}, \ \forall C \in C^S_k \label{cstr:cutset-MS}
  	\end{empheq}
  \end{subequations}

  The number of constraints (\ref{cstr:cutset-MS}) being exponential, we propose a constraints generation algorithm. We begin with a small number of constraints (\ref{cstr:cutset-MS}), associated with a small subset of $\mathcal{S}$. We obtain a lower bound for our problem. Then we select a cut-set that does not verify some constraint (\ref{cstr:cutset-cut}): given a network induced by the arcs $(i,j)$ such that $\hat{y}_{ij}=1$ (where $\hat{y}$ is the current value of $y$), we aim to find the cut of minimum residual capacity once we delete its $k$ most capacitated arcs. If this capacity is smaller than $|T|$, we add the constraints associated with this cut-set, otherwise the algorithm terminates. For small values of $k$, one straightforward method to find this cut of minimum residual capacity is the following: for each combination $C$ of $k$ arcs in $A_I$ which are selected in the current solution, compute the min-cut on the graph where the capacity of each arc $(i,j)$ is defined as $u_{ij} \hat{y}_{ij}$, except for the $k$ arcs of $C$ whose capacities are set to 0. Otherwise, the following MIP can be used:
  	
  	  \begin{subequations}
  		\begin{empheq}{align}
  		\min\limits_{s, d, v} & \quad \rlap{$\sum\limits_{(i,j) \in A} u_{ij}\hat{y}_{ij}s_{ij} $} \notag\\
  		\text{s.t} & \quad s_{ij} + d_{ij} - v_{i} + v_{j} & \geq \quad & 0 \quad & \forall (i,j) \in A \label{cstr:cutConstraint}\\
  		& \quad v_{r} & =  \quad& 1 \quad & \label{cstr:cutR}\\
  		& \quad v_{s} & =  \quad& 0 \quad & \label{cstr:cutS}\\
    	& \quad \sum\limits_{(i,j) \in A}   d_{ij} & \leq \quad& k \quad & \label{cstr:maxDel}\\
     	& \quad \sum\limits_{t \in T}   d_{ts} & = \quad& 0 \quad & \label{cstr:noDelFic}\\
     	& \quad \rlap{$s,d \ \in \ \{0,1\}^{|A|}, \qquad v \ \in \ \{0,1\}^{|V|}$}
  		\end{empheq}
  	\end{subequations}

  	 In this MIP, the variable $v$ defines a $r-s$-cut on the network: any vertex $j$ with $v_j=1$ is in the same part of the cut as $r$, and any vertex $i$ with $v_i=0$ is in the same part as $s$. The variable $d$ defines the deleted arcs, $d_{ij}=1$ if and only if arc $(i,j)$ is deleted, whereas $s$ defines the selected arcs, $s_{ij}=1$ if and only if arc $(i,j)$ is selected in the cut. Then, the constraints (\ref{cstr:cutConstraint}) ensure that $\hat{S}=\{(i,j)$ s.t.  $s_{ij}=1$ or   $d_{ij}=1  \}$ defines a cutset in the current network, and the objective function  $\sum_{(i,j) \in A} u_{ij}\hat{y}_{ij}s_{ij}$ represents the residual capacity of $\hat{S}$, i.e., the capacity of the undeleted arcs of $\hat{S}$.  the constraints (\ref{cstr:cutR}) and (\ref{cstr:cutS}) ensure that the root and the sink are not in the same part of the cut. The constraints (\ref{cstr:maxDel}) and (\ref{cstr:noDelFic}) ensure that there are no more than $k$ arc deletions, and that no fictive arc can be deleted. Notice that, in any optimal solution, the constraint (\ref{cstr:maxDel}) is saturated and, for each arc $(i,j)$, at most one of $s_{ij}$ and $d_{ij}$ is equal to 1.
  	  If the solution provides a cutset with a residual capacity at least equal to $|T|$, then the solution is feasible, and so optimal. Otherwise, we add the associated constraint to the main MIP.\\

	In the case of a uniform capacity $U$ on each arc $a \in A_I$, $M_S$ in the constraints (\ref{cstr:cutset-cut}) becomes a constant equal to $kU$ and hence these constraints are linear, and the constraints (\ref{cstr:cutset-MS}) are useless. The number of constraints is still exponential, but highly reduced compared to the non-uniform case. The formulation can be rewritten as follows:

  \begin{subequations}
  	\begin{empheq}{align}
  	\min\limits_{y}  & \quad  \sum\limits_{(i,j) \in A} c_{ij}y_{ij} \quad \notag \\
  	\text{s.t.}  & \sum\limits_{(i,j) \in S}u_{ij}y_{ij} \geq \ |T| \ + \ kU \ & \quad \forall S \in \mathcal{S} \quad \\
  	& \ \quad y_{ts} \quad = \quad 1 \qquad & \forall t \in T \quad \\
  	&  \ \quad y_{ij} \in \{0,1\} \qquad & \forall (i,j) \in A \quad \notag
  	\end{empheq}
  \end{subequations}

    Adapting the formulation to the undirected case where we are given a set of edges $E$ instead of arcs is quite straightforward. Indeed, it can be done by considering the undirected cut-sets of the graph instead of the directed ones.

  \subsubsection{Flow formulation} \label{subsubsec:flow}

  In this section, we introduce a formulation based on flow variables. We define $\mathcal{F}$ as the set of all possible arc-failure scenarios: it corresponds to the set of all $k$-combinations in $A_I$. We introduce the variable $x_{ij}^{F}$ which represents the amount of flow routed through the arc $(i,j) \in A$ when the scenario $F \in \mathcal{F}$ occurs. The variable $y$ is defined as in the previous formulation (see Subsection \ref{subsubsec:cut-set}). We propose the following flow formulation:

  \begin{subequations}
  	\begin{empheq}{align}
  	\min\limits_{x,y}   &  \sum\limits_{(i,j) \in A} c_{ij}y_{ij} \quad \notag\\[-3pt]
  	\text{s.t. }  & \hspace{-.2cm}\sum\limits_{i \in \Gamma^-(j)} \hspace{-.2cm} x^{F}_{ij} \ - \ \hspace{-.3cm}\sum\limits_{k \in \Gamma^+(j)} \hspace{-.2cm}x^{F}_{jk} = 0  & \forall j \in V \setminus \{r,s\}, \ \forall F \in \mathcal{F} \label{cstr:flow-flowCons} \\[-3pt]
  	& \hspace{-.2cm}\sum\limits_{t \in \Gamma^-(s)}\hspace{-.2cm} x^{F}_{ts} \quad  =  \quad |T|  &\forall F \in \mathcal{F} \label{cstr:flow-flowConsSink}\\
  	& x^{F}_{ij} \quad  \leq  \quad u_{ij}y_{ij}  & \forall (i,j) \in A, \ \forall F \in \mathcal{F} \label{cstr:flow-cap}\\
  	& x^{F}_{ij} \quad  =  \quad 0  & \forall F \in \mathcal{F}, \ \forall (i,j) \in F \label{cstr:flow-break}\\
  	&  \omit\rlap{$x \in \mathbb{R}_+^{|A| \times |\mathcal{F}|}, \quad y \in \{0,1\}^{|A|} $} & \notag
  	\end{empheq}
  \end{subequations}

  The constraints (\ref{cstr:flow-flowCons}) and (\ref{cstr:flow-flowConsSink}) ensure that there is a flow of value $|T|$ for each arc-failure scenario $F \in \mathcal{F}$, meaning that we can still route a unit of flow to each terminal after any $k$ arc failures. The constraints (\ref{cstr:flow-cap}) ensure that the arc capacities are respected for each arc-failure scenario $F \in \mathcal{F}$. The constraints (\ref{cstr:flow-break}) ensure that, in each scenario $F \in \mathcal{F}$, no flow is routed through deleted arcs. One can notice that the variable $x$ must be an integer (because it corresponds to a number of terminals). However, we relax this integrality constraint. Indeed, for any value of $y \in \{0,1\}^{|E|}$, setting the value of $x$ corresponds to routing a set of flows of value $|T|$ on $|E|$ different networks with integer capacities. Then, for any value of $y \in \{0,1\}^{|E|}$, there exists a solution where $x$ is integer. As only the variable $y$ is involved in the objective function, we have that there always exists an optimal solution with $x$ integer.\\

  The number of variables $x^{F}_{ij}$ and constraints (\ref{cstr:flow-flowCons}) and (\ref{cstr:flow-flowConsSink})  being exponential for arbitrary values of $k$, we propose a constraints-and-columns generation algorithm to solve the problem. We begin with a small subset of $\mathcal{F}$. The separation problem is the problem of the $k$ most vital links in a flow network \cite{ratliff1975finding}: we search for the $k$ arcs which, once simultaneously deleted, reduce the most the value of a maximum $s-t$ flow. We use a procedure similar to the one used in Subsection \ref{subsubsec:cut-set}: for small values of $k$, we compute a maximum $s-t$ flow for each combination of $k$ selected arcs of $A_I$. If there is a combination of arcs whose deletion results in a maximum $s-t$ flow lower than $|T|$, we add this arc-failure scenario, else the solution is optimal. If $k$ is too big, we use an auxiliary MIP.\\

  In order to adapt the formulation to the undirected case with a set of edges $E$ instead of a set of arcs $A$, one can define for each $[i,j] \in E$ the variables $y_{ij}$, $x_{ij}^F$ and $x_{ji}^F$. The function $\Gamma^+$ and $\Gamma^-$ are replaced by the function $\Gamma$ in the constraints (\ref{cstr:flow-flowCons}) and (\ref{cstr:flow-flowConsSink}), while the constraints (\ref{cstr:flow-cap}) and (\ref{cstr:flow-break}) are replaced by:

\begin{subequations}
	\begin{empheq}{align}
	 x^{F}_{ij} \ + \ x^{F}_{ji}  \quad  \leq  \quad u_{ij}y_{ij} \qquad & \forall [i,j] \in E, \ \forall F \in \mathcal{F} \\
	 x^{F}_{ij} \ + \ x^{F}_{ji}  \quad  =  \quad 0  \qquad & \forall F \in \mathcal{F}, \ \forall [i,j] \in F
	\end{empheq}
\end{subequations}

  A feasible solution induced by $y$ and $x$ implies a flow of value $|T|$ for each scenario $F \in \mathcal{F}$. Then, if a given solution yields a strictly positive flow on both $x_{ij}^F$ and $x_{ji}^F$ for a given edge $[i,j]$ and a given scenario $F$, it is trivial that there exists another flow at least as good as this one, but in which the amount of flow on either $x_{ij}^F$ or $x_{ji}^F$ is 0.

  \subsubsection{Bilevel formulation} \label{subsubsec:bilevel}

  The bilevel formulation proposed here is particular in that the second level is a $\min \max$ problem. It can be seen as a game with a defender and an attacker (corresponding respectively to the leader and the follower).\\

  For each $(i,j) \in A$, we introduce a variable $x_{ij}$ which corresponds to the amount of flow that the defender chooses to route through the arc $(i,j)$. The variable $y$ is defined as in Subsection \ref{subsubsec:cut-set}. We also introduce the binary variables  $b_{ij}$, $\forall (i,j) \in A$:  $b_{ij}= 1$ if and only if the attacker chooses to delete the arc $(i,j)$. Moreover, we assume without loss of generality that there is no arc of the form $(v,r)$ for some vertex $v$. Then, we can define the following polyhedron:

  \[
  \mathcal{X}(y,b)=\left\{
  x \in \mathbb{R}^{|A|} \left|
  \begin{array}{ll}
  \sum\limits_{i \in \Gamma^-(j)} x_{ij} -  \sum\limits_{k \in \Gamma^+(j)} x_{jk} = 0 \quad &\forall j \in V \setminus \{r,s\}\\
  x_{ij} \quad \leq \quad u_{ij}y_{ij} \ & \forall (i,j) \in A\\
  x_{ij} \quad \leq \quad u_{ij}(1 - b_{ij})  \ & \forall (i,j) \in A\\
  x_{ij} \quad \geq \quad 0 \ & \forall (i,j) \in A
  \end{array}
  \right.
  \right\}.
  \]\\

  This polyhedron $\mathcal{X}(y,b)$ corresponds to the set of possible flows on the subgraph of $G$ induced by the arcs $(i,j)$ such that $y_{ij}=1$, provided they have not been deleted (i.e. $b_{ij} = 0$). The polyhedron $\mathcal{X}(y,b)$ is defined by the flow conservation constraints, the capacity constraints and the constraints imposing a flow equal to 0 on any arc which is deleted. We also define the following polyhedrons:

  \begin{center}
	$\mathcal{B} = \{ \ b \in \{0,1\}^{|A|} \ | \ \sum_{(i,j) \in A}  b_{ij} \leq k \ ; \ b_{ts} = 0 \quad \forall t \in T\ \}$\\
  	\vspace{5pt}
  \end{center}

  The polyhedron $\mathcal{B}$ defines the set of possible scenarios of arc failures (it ensures that no fictive arc can be deleted). We propose the following bilevel program:

  \begin{subequations}
  	\begin{empheq}{align}
  	\min\limits_{y \in  \{0,1\}^{|A|}} \quad &  \sum\limits_{(i,j) \in A} c_{ij}y_{ij} &&&& \notag\\
  	\text{s.t.} \quad & f(y) \geq |T| &&&& \label{fYCstr}\\
  	&\text{where } f(y) = & \hspace{-.2cm} \min\limits_{b \in \mathcal{B}} \quad & \max\limits_{x \in \mathcal{X}(y,b)} &&  \hspace{-.3cm} \sum\limits_{j \in \Gamma^+(r)} x_{rj}&&
  	\end{empheq}
  \end{subequations}\\

where  $\{(i,j)$ s.t. $y_{ij}=1\}$ defines the set of selected arcs. At the upper level, the defender selects the set of arcs to be added to the network, by choosing a value of $y$ in $ \{0,1\}^{A}$. The attacker then deletes some arcs by setting the variable $b \in \mathcal{B}$ in order to minimize the maximum flow that the defender will compute by setting the variable $x$ in the flow polyhedron $\mathcal{X}(y,b)$. The aim of the defender is to ensure that this flow is at least equal to $|T|$ (see constraint (\ref{fYCstr})).\\

  Consider the $\max$ problem in the lower level: at this stage, $y$ and $b$ are already fixed; we refer to their values as $\hat{y}$ and $\hat{b}$ respectively. The problem is a max-flow problem from $r$ to $s$, with two sets of capacity constraints. In our problem, the flow must be integral since it corresponds to a number of terminals. However, it is well-known that the matrix of coefficients $M$ in the arc-formulation of a max-flow is totally unimodular. Then, adding the second set of capacity constraints is equivalent to appending the identity matrix to $M$: the matrix remains totally unimodular and, as the capacities are integers, we ensure that the extreme points of the polyhedron defined by $\mathcal{X}(y,b)$ have integral coordinates. Thus, we can relax the integrality constraints on $x$.\\

  In this $\max$ problem of the lower level, there always exists a feasible flow of value 0 and the problem is also trivially upper bounded. Hence, the strong duality holds and we can introduce the dual of the lower level problem, after a slight reformulation due to the totally unimodular matrix:

  \begin{subequations}
  	\begin{empheq}{align}
  	\min\limits_{\lambda, \mu, \gamma} & \quad \rlap{$\sum\limits_{(i,j) \in A} u_{ij}\hat{y}_{ij}\lambda_{ij} + \sum\limits_{(i,j) \in A} u_{ij}(1 - \hat{b}_{ij})\gamma_{ij}$} \notag\\
  	\text{s.t} & \quad \lambda_{ij} + \gamma_{ij} - \mu_{i} + \mu_{j} & \geq \quad & 0 \quad & \forall (i,j) \in A \label{lambdaGammaCstr}\\
  	& \quad \mu_{r} & =  \quad& 1 \quad & \label{muRconstraints}\\
  	& \quad \mu_{s} & =  \quad& 0 \quad & \label{muSConstraints}\\
  	& \quad \rlap{$\lambda,\gamma \ \in \ [0,1]^{|A|}, \qquad \mu \ \in \ [0,1]^{|V|}$} \label{3L-varCstr}
  	\end{empheq}
  \end{subequations}

  This problem is a special formulation of a min-cut problem: $\mu$ defines the two parts of the cut (sets of vertices $i \in V$ such that either $\mu_i =0$ or $\mu_i =1$). The variables $\gamma$ and $\lambda$ define the cut-set of the corresponding cut: for each arc $(i,j)$ in the cut-set, we have either $\lambda_{ij}=1$ or $\gamma_{ij} = 1$, otherwise we have $\gamma_{ij}=\lambda_{ij}=0$. Because of the economic function and the positive capacities, we have that $\gamma_{ij}$ is equal to 1 for at least all arcs $(i,j)$ in the cut-set with $\hat{b}_{ij} = \hat{y}_{ij} = 1$ (i.e., the arcs selected but deleted), while $\lambda_{ij}$ is equal to 1 for at least all arcs $(i,j)$ in the cut-set with $\hat{b}_{ij} = \hat{y}_{ij} = 0$ (i.e., the arcs that are neither selected nor deleted). For other arcs in the cut-set, it does not matter which one is set to 1. We denote by $\mathcal{D}$ the polyhedron defined by the dual constraints (\ref{lambdaGammaCstr})$-$(\ref{3L-varCstr}).\\

  As the lower level can be reformulated as a $\min-\min$ function by using the dual described above, it can then be rewritten as follows:

  \begin{equation*}(2LP) \left |
  \begin{array}{ll}
  \min\limits_{b,\lambda, \mu, \gamma} & \quad \sum\limits_{(i,j) \in A} u_{ij}\hat{y}_{ij}\lambda_{ij} + u_{ij}(1 - b_{ij})\gamma_{ij} \notag\\
  \text{s.t} \quad & \quad b \in \mathcal{B}\\
  & \quad (\lambda, \mu, \gamma) \in \mathcal{D}
  \end{array}
  \right.
  \end{equation*}

  At this point, $b$ is a variable, so the objective function is non-linear. We linearize the terms $b_{ij} \gamma_{ij}$ in a classical way  by introducing binary variables $l_{ij}$ verifying the set of constraints  defined by $\mathcal{L}(b,\gamma)$:

    \[
    \mathcal{L}(b,\gamma)=\left\{
    l \in \mathbb{R}^{|A|} \left|
    \begin{array}{ll}
    l_{ij} \quad \leq \quad b_{ij} \ & \forall (i,j) \in A\\
    l_{ij} \quad \leq \quad \gamma_{ij}  \ & \forall (i,j) \in A\\
    l_{ij} \quad \geq \quad \gamma_{ij} - (1 - b_{ij}) \ & \forall (i,j) \in A\\
        l_{ij} \quad \geq \quad 0 \ & \forall (i,j) \in A
    \end{array}
    \right.
    \right\}
    \]\\

  We also define the function $g(y,\lambda,\gamma, l) =  \sum_{(i,j) \in A} \left[u_{ij}y_{ij}\lambda_{ij} + u_{ij}\gamma_{ij} - u_{ij}l_{ij}\right]$. We can then rewrite the bilevel program as:

    \begin{equation*}
    \begin{array}{lllll}
    \min\limits_{y \in  \{0,1\}^{|A|}} & \rlap{$\sum\limits_{(i,j) \in A} c_{ij}y_{ij} $} \notag\\
    \text{s.t} \quad & \quad f(y) \geq |T| && \label{2L-fYCstr}\notag\\
	 & \qquad\text{where }\quad & f(y) \quad = & \min\limits_{b, \lambda, \gamma, \mu, l} & \rlap{$  g(y,\lambda,\gamma, l)$}\\
	 &  &  & \quad \text{s.t. } \quad& b \in \mathcal{B}  \notag\\
	 	&  &  \quad& & (\lambda, \mu, \gamma) \in \mathcal{D}  \notag\\
	 	&  &  \quad&  &l \in \mathcal{L}(b,\gamma) \notag
    \end{array}
    \end{equation*}

  We can then consider the convex hull of the lower-level polyhedron, and denote by $\mathcal{H}$ the set of its extreme points. One can notice that this convex hull does not depend on $y$ (only $g(\cdot)$ does): the set of extreme points $\mathcal{H}$ remains the same for every $y \in  \{0,1\}^{A}$. We denote by  $(\hat{\lambda}^h, \hat{\gamma}^h, \hat{l}^h)$ the respective values of $(\lambda, \gamma, l)$ at the extreme point $h \in \mathcal{H}$. We can then reformulate the bilevel formulation as a single-level one as follows:

  \begin{subequations}
  	\begin{empheq}{align}
  	\min    &  \sum\limits_{(i,j) \in A} c_{ij}y_{ij} &\notag\\
  	\text{s.t.} \quad & g(y,\hat{\lambda}^h,\hat{\gamma}^h, \hat{l}^h)  \ \geq \ |T| & \forall h \in \mathcal{H}\label{cstr:biGA} \\
  	&  y \in  \{0,1\}^{|A|} & \label{cstr:polyYP}\\
  	\mathbf{(BP)} \hspace{2cm}&  b \in \mathcal{B} & \label{cstr:polyB}\\
  	&  (\lambda, \mu, \gamma) \in \mathcal{D} \label{cstr:polyD}\\
  	&  l \in \mathcal{L}(b,\gamma) & \label{cstr:polyL}&
  	\end{empheq}
  \end{subequations}

  The constraints (\ref{cstr:biGA}) ensure that, for each extreme point of $\mathcal{H}$, $f(y)$ is greater than $|T|$ (i.e., the minimum value of $f(y)$ over the polyhedron defined by the constraints (\ref{cstr:polyB})$-$(\ref{cstr:polyD}) is greater than $|T|$), meaning that the value of a maximum flow cannot become smaller than $|T|$, even after any $k$ breakdowns.

  \begin{rmk}
  	In $\mathbf{(BP)}$, $g(y,\lambda,\gamma, l)$ is non-linear because of the products $y_{ij}\lambda_{ij}$, but they can be linearized as it has been done for  $b_{ij} \gamma_{ij}$ above.
  \end{rmk}

  However, there is an exponential number of constraints (\ref{cstr:biGA}), and we do not know how to describe explicitly the convex hull of $\mathcal{H}$. To tackle this issue, we use a constraints generation algorithm where we relax the constraints (\ref{cstr:biGA}) and use $(2LP)$ as the separation problem: while the optimum value of $(2LP)$ is smaller than $|T|$ for the current optimal solution $\hat{y}$, we generate the constraints (\ref{cstr:biGA}) associated with the extreme point whose coordinates are the optimal values of $(b, \lambda, \gamma, \mu, l)$ in $(2LP)$.

  \begin{prop}\label{betterCutProp}
  	Let $\hat{y}^{1}$ and $\hat{y}^{2}$ be two feasible solutions of $\mathbf{(BP)}$ such that $\hat{y}^{1} \geq \hat{y}^{2}$, i.e., $\hat{y}^{1}_{ij} \geq \hat{y}^{2}_{ij}$ for each arc $(i,j)$. If adding a constraint $g(y,\lambda,\gamma,l) \leq g(y,\hat{\lambda}^a, \hat{\gamma}^a, \hat{l}^a)$ makes any solution with $y = \hat{y}^{1}$ infeasible, then it also makes any solution with $y = \hat{y}^{2}$ infeasible.
  \end{prop}

  \begin{proof}
  	For any value $(\hat{\lambda}^a, \hat{\gamma}^a, \hat{l}^a)$ of $(\lambda, \gamma, l)$, we have $g(\hat{y}^{1},\hat{\lambda}^a, \hat{\gamma}^a, \hat{l}^a) \geq g(\hat{y}^{2},\hat{\lambda}^a, \hat{\gamma}^a, \hat{l}^a)$ since $\hat{y}^{1} \geq \hat{y}^{2}$ (as $u$ and $\lambda$ are positive). Hence, if $g(\hat{y}^{1},\hat{\lambda}^a, \hat{\gamma}^a, \hat{l}^a) \leq |T| - 1$, then  $g(\hat{y}^{2},\lambda^a,\gamma^a,l^a) \leq |T| - 1$.
  \end{proof}

  To improve the cut obtained by solving $(2LP)$ at each step, we try to inject better values $\hat y$ of variables $y$ in it. To get these values, we first solve the following problem, and then we compute the new $\hat y$ accordingly (as explained later). Given a starting solution $\hat{y}$, we want to find a cut-set in the support network (i.e., in the initial digraph $G$) with a minimum number of arcs such that this cut-set is non-valid in the network induced by the arcs $(i,j)$ such that $\hat{y}_{ij}=1$ (meaning that, if we remove $k$ given arcs of the cut-set, its remaining capacity is smaller than $|T|$). This can be modeled as follows:
  \vspace{-0.3cm}
  \begin{subequations}
  	\begin{empheq}{align}
  	\min  \quad  &  \rlap{$\sum_{(i,j) \in A} \lambda_{ij}$} \notag\\[-4pt]
  	\text{s.t} \quad& \sum_{(i,j) \in A} u_{ij}\hat{y}_{ij}\lambda_{ij}  & \leq & \quad |T| - 1 \label{cstr:enhanceNF}\\[-4pt]
  	& \sum_{(i,j) \in A} \gamma_{ij} & \leq & \quad k & \label{cstr:enhanceK}\\[-4pt]
  	&\ \gamma_{ts} & = & \quad 0 \quad & \forall t \in T \label{cstr:enhanceFic}\\
  	& \omit\rlap{$\quad (\lambda, \mu, \gamma) \in \mathcal{D}, \ \mu \in \{0,1\}^{|V|} $} &&&\notag
  	\end{empheq}
  \end{subequations}

  The variables $(\lambda, \mu, \gamma)$ define a cut as in $(2LP)$ since they belong to $\mathcal{D}$ (recall that $\mathcal{D}$ is the set of constraints (\ref{lambdaGammaCstr})$-$(\ref{3L-varCstr})). However, adding the other constraints makes the constraints matrix not unimodular anymore: thus, we have to set $\mu$ as a 0-1 variable. The constraint (\ref{cstr:enhanceNF}) ensures that the cut-set selected is non-valid (as defined before). The constraint (\ref{cstr:enhanceK}) bounds the number of deleted arcs to at most $k$, while the constraints (\ref{cstr:enhanceFic}) forbid the deletion of fictive arcs. \\

  Then, the new values of the $\hat y_{ij}$'s are computed as follows: we set $\hat y_{ij}$ to 1 for all $(i,j)$ with $\lambda_{ij} = \gamma_{ij} = 0$ and let the others to their current value. It implies that, for each arc $(i,j)$, the new value of $\hat y_{ij}$ cannot be smaller than the old one, and, using Proposition \ref{betterCutProp}, we generate a better constraint than the original one by computing the extreme points associated with this new value of $\hat y$. \\

  In order to obtain a formulation that works for the undirected case, we define for each edge $[i,j]$ the variables $y_{ij}$, $b_{ij}$, $x_{ij}$ and $x_{ji}$. The only modification appears in the polyhedron $\mathcal{X}(y,b)$, which can be modified as follows:

   \[
   \mathcal{X}(y,b)=\left\{
   x \in \mathbb{R}^{|E|} \left|
   \begin{array}{ll}
   \sum\limits_{i \in \Gamma^(j)} x_{ij} -  \sum\limits_{k \in \Gamma^(j)} x_{jk} = 0 \quad &\forall j \in V \setminus \{r,s\}\\
   x_{ij} \quad \leq \quad u_{ij}y_{ij} \ & \forall [i,j] \in E\\
   x_{ji} \quad \leq \quad u_{ij}y_{ij} \ & \forall [i,j] \in E\\
   x_{ij} \quad \leq \quad u_{ij}(1 - b_{ij})  \ & \forall [i,j] \in E\\
   x_{ji} \quad \leq \quad u_{ij}(1 - b_{ij})  \ & \forall [i,j] \in E\\
   x_{ij}, x_{ji} \quad \geq \quad 0 \ & \forall [i,j] \in E
     \end{array}
   \right.
   \right\}
   \]\\

   Again, as this polyhedron is associated with a maximum flow problem (when the values of $y$ and $b$ are fixed), we can always find a maximum flow where either $x_{ij} = 0$ or $x_{ji} = 0$ for each edge $[i,j] \in E$. Once this polyhedron has been modified, one can use the method proposed for the directed case to solve the formulation.

\subsection{Addition of protected arcs}

Let us now define another version of the problem, where we add the possibility of protecting $k'$ arcs. In this version, in addition to $A'$, we also select a subset $A'_p \subset A'$ with $|A'_p| \leq k'$; those arcs are called \emph{protected arcs} and cannot be deleted by the attacker. The corresponding problem is called Capacitated Protected Rooted k-Edge Connected Steiner Network problem (\textbf{CPRkECSN}). In the wind farm application, protecting arcs can be seen as doubling a set of cables under a given budget for example, or protecting cables from a difficult environment (like extreme cold). \\

For each arc $(i,j)$, we define the variable $p_{ij}$ as a binary variable equal to 1 if the arc $(i,j)$ is protected, and to 0 otherwise. We also define the set of values that can be taken by $p$:

\begin{center}
	 $P = \{ \ p \in \{0,1\}^{|A|} \ | \sum\limits_{(i,j) \in A}  p_{ij} \leq k' \ ; \ p_{ij} \ \leq \ y_{ij} \quad \forall (i,j) \in A \ \}$
\end{center}

This set ensures that there are at most $k'$ protected arcs, and that we cannot protect arcs which are not selected in the final network. In the following, we propose small modifications to each one of the previous formulations in order to include the possibility of protecting arcs.

\subsubsection{Cut-set formulation}

In the cut-set formulation, the constraints (\ref{cstr:cutset-MS}) can be replaced by the following ones:

  \begin{subequations}
  	\begin{empheq}{align}
  	\sum\limits_{(i,j) \in S} u_{ij}y_{ij} \ - \ \sum_{(i,j) \in C} u_{ij}(y_{ij} - p_{ij}) \quad \geq \quad |T| \qquad \forall S \in \mathcal{S}, \ \forall C \in C^S_k \label{cstr:cutsetProt-MS}
  	\end{empheq}
  \end{subequations}

  We check that the capacity of each cut-set minus the capacity of $k$ unprotected arcs of this cut-set is always larger than $|T|$. We also add to the cut-set formulation the constraint $p \in P$. We solve the resulting MIP using the same constraints generation algorithm as in Subsection \ref{subsubsec:cut-set}. The separation problem is slightly modified to take into account the fact that the capacity of the protected arcs cannot be removed to compute the residual capacity of the cut-set. For small values of $k$, for each combinations of $k$ selected but non-protected arcs, we compute the min-cut (in Subsection \ref{subsubsec:cut-set}, we take into account all selected arcs). Considering the MIP, we just have to add the constraint $d_{ij} \leq 1 - \hat p_{ij}$ for each arc $(i,j)$ (where $\hat p$ corresponds to the current value of $p$).

  \begin{rmk}
  When arcs can be protected, the case of uniform capacities does not admit a simpler formulation anymore.
  \end{rmk}

\subsubsection{Flow formulation}

In the flow formulation, in addition to the constraint $p \in P$, we can replace the constraints (\ref{cstr:flow-break}) by the following ones:

\begin{subequations}
	\begin{empheq}{align}
	x^{F}_{ij} \quad \leq \quad u_{ij}p_{ij}  \qquad \forall F \in \mathcal{F}, \ \forall (i,j) \in F \label{cstr:flowProt-break}
	\end{empheq}
\end{subequations}

Those constraints ensure that in a scenario $F$ where an arc $(i,j) \in F$, we can route some flow through this arc $(i,j)$ only if this arc is protected. Again, we can use the same columns-and-constraints generation algorithm as in Subsection \ref{subsubsec:flow}, in order to find the most vital arcs in the separation problem among the non-fictive and non-protected arcs (we consider only combinations of selected but non-protected arcs when computing the set of maximum flows).

\subsubsection{Bilevel formulation}

In the bilevel formulation, the only polyhedron that needs to be modified is $\mathcal{X}(y,b)$, which is replaced by the following one, denoted by $\mathcal{X}(y,b,p)$:

  \[
  \mathcal{X}(y,b,p)=\left\{
  x \in \mathbb{R}^{|A|} \left|
  \begin{array}{ll}
  \sum\limits_{i \in \Gamma^-(j)} x_{ij} -  \sum\limits_{k \in \Gamma^+(j)} x_{jk} = 0 \quad &\forall j \in V \setminus \{r,s\}\\
  x_{ij} \quad \leq \quad u_{ij}y_{ij} \ & \forall (i,j) \in A\\
  x_{ij} \quad \leq \quad u_{ij}(1 - b_{ij} + p_{ij})  \ & \forall (i,j) \in A\\
  x_{ij} \quad \geq \quad 0 \ & \forall (i,j) \in A
  \end{array}
  \right.
  \right\}
  \]\\

  The third constraint ensures that, if an arc is protected, then we can route a flow through this arc (respecting the capacities) even if the attacker deletes it. The bilevel formulation for the problem with protected arcs is then:

 \begin{equation*}
    \begin{array}{lrllll}
    &\min\limits_{y \in  \{0,1\}^{|A|}, p \in P} & \rlap{$\sum\limits_{(i,j) \in A} c_{ij}y_{ij} $} \notag\\
    &\text{s.t} \quad & \quad f(y,p) \geq |T| && \label{2L-fYCstr-prot}\notag\\
	& & \qquad\text{where }\quad & f(y,p) \quad = & \min\limits_{b, \lambda, \gamma, \mu, l} & \rlap{$  g_{prot}(y,p,\lambda,\gamma, l)$}\\
	\mathbf{(BPP)} \hspace{2cm} & &  &  & \quad \text{s.t. } \quad& b \in \mathcal{B}  \notag\\
	 	& &  &  \quad& & (\lambda, \mu, \gamma) \in \mathcal{D}  \notag\\
	 	& &  &  \quad&  &l \in \mathcal{L}(b,\gamma) \notag
    \end{array}
    \end{equation*}

where $g_{prot}(y,p,\lambda,\gamma, l) =  \sum_{(i,j) \in A} u_{ij}y_{ij}\lambda_{ij} + u_{ij}\gamma_{ij} - u_{ij}l_{ij} + u_{ij}p_{ij}\gamma_{ij}$. We then use the same decomposition method to solve the formulation. Property \ref{betterCutProp} can be replaced by the following one (using the fact that $u \geq 0$ and $\gamma \geq 0$):

\begin{prop}\label{betterCutPropProt}
  	Let $(\hat{y}^{1},\hat{p}^{1})$ and $(\hat{y}^{2},\hat{p}^{2})$ be two feasible solutions of $\mathbf{(BPP)}$ such that $\hat{y}^{1} \geq \hat{y}^{2}$ and $\hat{p}^{1} \geq \hat{p}^{2}$. If adding a constraint $g_{prot}(y,p,\lambda,\gamma,l) \leq g_{prot}(y,p,\hat{\lambda}^a, \hat{\gamma}^a, \hat{l}^a)$ makes any solution with $(y,p) = (\hat{y}^{1},\hat{p}^{1})$ infeasible, then it also makes any solution with $(y,p) = (\hat{y}^{2},\hat{p}^{2})$ infeasible.
  \end{prop}

The same MIP can be used to enhance the generated constraint, simply by replacing the constraint (\ref{cstr:enhanceNF}) by the following one:

\[\sum_{(i,j) \in A} u_{ij}\hat{y}_{ij}\lambda_{ij} + \hat{p}_{ij}\gamma_{ij} \quad \leq \quad |T| - 1 \]

\subsection{Valid inequalities}\label{subsec:validIneq}

In this section, we propose some valid inequalities to enhance the quality of the lower bound obtained by solving the continuous relaxation.

\begin{subequations}
	\begin{empheq}{align}
	\sum_{(i,t) \in A} y_{it} \quad \geq \quad k + 1 \qquad \forall t \in T \label{VI:termNonProtected}
	\end{empheq}
\end{subequations}

\begin{subequations}
	\begin{empheq}{align}
	\sum_{(r,i) \in A} y_{ri} \quad \geq \quad k + 1 \label{VI:rootNonProtected}
	\end{empheq}
\end{subequations}

Inequalities (\ref{VI:termNonProtected}) ensure that there are at least $k+1$ arcs entering each terminal. Indeed, if there are less than $k+1$ arcs entering it, then it is possible to delete all of them and thus to prevent one unit of flow from reaching the sink. Inequality (\ref{VI:rootNonProtected}) states the same constraint for the root. Those two families of inequalities are only true for the case without protection ($k'=0$). As one arc can be enough if it is protected, we can replace the previous inequalities by (\ref{VI:termProtected}) and (\ref{VI:rootProtected}) in this case.

\begin{subequations}
	\begin{empheq}{align}
	\sum_{(i,t) \in A} y_{it} \quad \geq \quad 1 \qquad \forall t \in T \label{VI:termProtected}
	\end{empheq}
\end{subequations}

\begin{subequations}
	\begin{empheq}{align}
	\sum_{(r,i) \in A} y_{ri} \quad \geq \quad 1 \label{VI:rootProtected}
	\end{empheq}
\end{subequations}

Inequalities (\ref{VI:SteinerIn}) state that, for each Steiner vertex $j$, if an arc entering $j$ is selected, then at most one arc leaving $j$ must be selected. Indeed, in a feasible solution, if there is no arc leaving a Steiner vertex, then all arcs entering it can be deleted without making the solution infeasible (as we assumed the arc costs to be positive, this new solution is at least as good as the previous one). Inequalities (\ref{VI:SteinerOut}) state the same for arcs leaving a Steiner vertex $j$. Those inequalities are valid in both protected and unprotected versions of the problem.

\begin{subequations}
	\begin{empheq}{align}
	y_{ij} \quad \leq \quad \sum_{k \in \Gamma^+(j)} y_{jk} \qquad \forall j \in V \setminus \{T \cup \{r\}\}, \ \forall i \in \Gamma^-(j) \label{VI:SteinerIn}
	\end{empheq}
\end{subequations}

\begin{subequations}
\begin{empheq}{align}
y_{jk} \quad \leq \quad \sum_{i \in \Gamma^-(j)} y_{ij} \qquad \forall j \in V \setminus \{T \cup \{r\}\}, \ \forall k \in \Gamma^+(j) \label{VI:SteinerOut}
\end{empheq}
\end{subequations}

Inequalities (\ref{VI:SteinerNeigh}) relate to the Steiner vertices which are adjacent to a terminal, in the unprotected version of the problem. Terminals must have at least $k+1$ neighbors in a feasible solution. Any given terminal $t \in T$ has a number of neighbors which are terminals equal to $|\Gamma_{G}(t) \cap T|$. Then, there must be at least $\max(0,(k+1)-|\Gamma_{G}(t) \cap T|)$ neighbors of $t$ which are Steiner vertices in a feasible solution. As any Steiner vertex selected in an inclusion-wise minimal solution has at least two incident arcs, we ensure that the number of those arcs (counted with their multiplicity) is larger than two times the number of necessary Steiner vertices adjacent to $t$.

\begin{subequations}
	\begin{empheq}{align}
	\sum_{j \in \Gamma_{G}(t) \cap (V \setminus T)} (\sum_{i \in \Gamma_{G}^-(j)} y_{ij} + \sum_{k \in \Gamma_{G}^+(j)} y_{jk}) \quad \geq \quad 2(k+1-|\Gamma_{G}(t) \cap T|) \qquad \forall t \in T \label{VI:SteinerNeigh}
	\end{empheq}
\end{subequations}

In the case with protected arcs, each vertex $v \in T \cup \{r\}$ must have at least $k+1$ arcs entering it (or leaving it, if $v = r$), except if it has at least one protected arc entering it (or leaving it, if $v = r$), which can happen for at most $k'$ vertices among these $|T|+1$. Hence, we obtain:

\begin{subequations}
	\begin{empheq}{align}
	\sum\limits_{t \in T}\sum\limits_{i \in \Gamma^-_{G}(t)} y_{it} \ + \ \sum\limits_{j \in \Gamma^+_{G}(r)} y_{rj}  \quad \geq \quad (|T|+1-k')(k+1) + k'
	\end{empheq}
\end{subequations}

\section{Results analysis}\label{sec:results}

In this section, we present the results of the three formulations proposed previously. All experiments were performed on a computer with a 2.40GHz Intel(R) Core(TM) i7-5500U CPU and a 16GB RAM, using the solver CPLEX version 12.6.1, interfaced with Julia 0.6.0. We used in particular the package \emph{JuMP}, a tool allowing mathematical modeling. For each test, the algorithm has been stopped after 3000 seconds if it has not terminated yet. Table \ref{tab:parametersSteiner} shows for each instance the number associated ($I$), as well as the number of vertices, terminals, and edges, respectively. The column $opt$ gives the optimal value of the Capacitated Rooted Steiner Network (\textbf{CRSN}) for each instance, which corresponds to the case where $k=0$. All instances have been generated in the following way: the vertices have been generated in the plane, and the capacity of an arc is more likely to be high if this arc is close to the root. The arc capacities are high enough to have a set of feasible solutions to our problems, but low enough to keep the difficulty in those problems. More precisely, the capacities are chosen randomly among four values: $0.8|T|$, $0.6|T|$ and, except for the edges with endpoints at distance 1 or 2 from the root, $0.4|T|$ and $0.2|T|$. Furthermore, the cost of an arc depends on both its length and its capacity, and hence is not necessarily integral.\\

\begin{table}[h!]
	\begin{subtable}[h]{0.30\linewidth}
		\centering
		\scalebox{0.8}{\begin{tabular}{|r|r|r|r|r|}
			\hline
			I & |V| & |T| & |E| & opt  \\
			\hline
			1 & 20 & 2 & 47 & 3.43 \\
			\hline
			2 & 20 & 4 & 46 & 5.16 \\
			\hline
			3 & 20 & 6 & 45 & 7.8 \\
			\hline
			4 & 20 & 12 & 46 & 10.73 \\
			\hline
			5 & 20 & 19 & 46 & 14.14 \\
			\hline
			6 & 25 & 2 & 60 & 4.62 \\
			\hline
			7 & 25 & 5 & 59 & 5.78 \\
			\hline
			8 & 25 & 8 & 61 & 8.11 \\
			\hline
			9 & 25 & 15 & 61 & 10.36 \\
			\hline
			10 & 25 & 24 & 59 & 13.93 \\
			\hline
			11 & 30 & 18 & 74 & 9.45 \\
			\hline
			12 & 30 & 3 & 74 & 3.83 \\
			\hline
		\end{tabular}}
	\end{subtable}
	\hfill
	\begin{subtable}[h]{0.30\linewidth}
		\centering
		\scalebox{0.8}{\begin{tabular}{|r|r|r|r|r|}
			\hline
			I & |V| & |T| & |E| & opt  \\
			\hline
			13 & 30 & 6 & 73 & 10.52 \\
			\hline
			14 & 30 & 9 & 74 & 7.19 \\
			\hline
			15 & 30 & 29 & 74 & 13.85 \\
			\hline
			16 & 35 & 4 & 89 & 3.17 \\
			\hline
			17 & 35 & 7 & 87 & 7.51 \\
			\hline
			18 & 35 & 10 & 91 & 8.95 \\
			\hline
			19 & 35 & 21 & 89 & 11.58 \\
			\hline
			20 & 35 & 34 & 88 & 12.81 \\
			\hline
			21 & 40 & 4 & 104 & 6.0 \\
			\hline
			22 & 40 & 8 & 103 & 8.83 \\
			\hline
			23 & 40 & 12 & 100 & 12.73 \\
			\hline
		\end{tabular}}
	\end{subtable}
	\hfill
	\begin{subtable}[h]{0.30\linewidth}
		\centering
		\scalebox{0.8}{\begin{tabular}{|r|r|r|r|r|}
			\hline
			I & |V| & |T| & |E| & opt  \\
			\hline
			24 & 40 & 24 & 104 & 10.93 \\
			\hline
			25 & 40 & 39 & 103 & 14.95 \\
			\hline
			26 & 45 & 14 & 118 & 10.26 \\
			\hline
			27 & 45 & 27 & 114 & 17.03 \\
			\hline
			28 & 45 & 44 & 119 & 18.74 \\
			\hline
			29 & 45 & 4 & 119 & 3.41 \\
			\hline
			30 & 50 & 5 & 133 & 6.48 \\
			\hline
			31 & 50 & 10 & 133 & 9.75 \\
			\hline
			32 & 50 & 15 & 131 & 8.22 \\
			\hline
			33 & 50 & 30 & 133 & 12.29 \\
			\hline
			34 & 50 & 49 & 130 & 16.05 \\
			\hline
		\end{tabular}}
	\end{subtable}
	\caption{Instance parameters and results of \textbf{CRSN}} \label{tab:parametersSteiner}
\end{table}

\begin{table}[h!]
	\centering
	\scalebox{0.7}{\begin{tabular}{|r|r?r|r|r|r?r|r|r|r?r|r|r|r|}
			\hline
			\multicolumn{2}{|c?}{Parameters} & \multicolumn{4}{c?}{Bilevel} & \multicolumn{4}{c?}{Cut-set} & \multicolumn{4}{c|}{Flow}\\
			\hline
			k & I & $\%_{opt}$ & time(s) & CRG & it & $\%_{opt}$ & time(s) & CRG & it & $\%_{opt}$ & time(s) & CRG & it\\
			\Xhline{1.5pt}  
			1 & $I_1$ & 100.0 & 18.24 & 0.3 & 159 & 100.0 & 31.93 & 0.17 & 91 & 100.0 & 18.42 & 0.17 & 12 \\
			\hline
			- & $I_2 \setminus \{13,19\}$ & 93.33 & 382.11 & 0.26 & 425 & 60.0 & 1396.57 & 0.17 & 389 & - & - & - & - \\
			\hline
			2 & 2,3,6,8 & 100.0 & 6.9 & 0.21 & 72 & 100.0 & 33.9 & 0.08 & 66 & 100.0 & 69.28 & 0.08 & 30 \\
			\hline
			- & 14,18,21,23,25,26 & 100.0 & 209.3 & 0.22 & 428 & 50.0 & 1692.22 & 0.16 & 253 & - & - & - & - \\
			\hline
			3 & 6,8 & 100.0 & 16.33 & 0.17 & 140 & 100.0 & 35.14 & 0.13 & 89 & 100 & 246.45 & 0.13 & 82 \\
			\hline
			- & 18,21 & 100.0 & 199.25 & 0.26 & 418 & 50.0 & 2472.0 & 0.17 & 176 & - & - & - & - \\
			\hline                                                                                  
	\end{tabular}}
	\caption{Results on instances with non-uniform capacities and $k'=0$ } \label{tab:ResultsNonUnifK0}
\end{table}

  Table \ref{tab:ResultsNonUnifK0} shows the results obtained, for $k \in \{1,2,3\}$ and $k'=0$, by the three formulations on those instances with non-uniform capacities. For each value of $k$, we have computed the results on two subsets of instances, one called $I_1$ on which we test the three formulations, one called $I_2$ composed of instances of larger size, on which we test only the bilevel and the cut-set formulations as  the solving time of the flow formulation is too important in that case. The subset $I_1$ is composed of the instances numbered from 1 to 11 while $I_2$ is composed of the ones numbered from 12 to 28. As some instances do not admit feasible solutions for some values of $k$, we remove such instances. Nevertheless, those instances are taken into account in the results when $k' \neq 0$ and in the case of uniform capacities. The column $k$ is the number of arc deletions considered for this subset of instances, while $I$ gives the numbers of the instances tested. For each formulation, $\%_{opt}$ corresponds to the percentage of instances solved to optimality, $time$ shows the mean solving time (in seconds) for this subset, $CRG$ corresponds to the mean gap between the optimal value of the continuous (or linear) relaxation and the optimal value of the problem and $it$ corresponds to the mean value of iterations. As we can see in this table, on small instances and small values of $k$, the three formulations can be competitive. However, as the sizes of the instances and value of $k$ grow, the bilevel formulation tends to be the most efficient one, while the flow formulation seems to be the less effective one, despite the fact that the optimal values of the continuous relaxation of the flow and cut-set formulations yield the same value on this set of instances, and are better than the one of the bilevel formulation in this case. We can also notice that, on those instances, the connectivity requirements have a high impact on the cost of an optimal solution, as this cost raises consequently as $k$ increases.\\

\begin{table}[h!]
	\centering
	\scalebox{0.7}{\begin{tabular}{|r|r|r?r|r|r|r?r|r|r|r?r|r|r|r|}
			\hline
			\multicolumn{3}{|c?}{Parameters} & \multicolumn{4}{c?}{Bilevel} & \multicolumn{4}{c?}{Cut-set} & \multicolumn{4}{c|}{Flow}\\
			\hline
			k & k' & I & $\%_{opt}$ & time(s) & CRG & it & $\%_{opt}$ & time(s) & CRG & it & $\%_{opt}$ & time(s) & CRG & it\\
			\Xhline{1.5pt}  
			1 & 1 & $I_1$ & 100.0 & 56.62 & 0.41 & 251 & 90.0 & 361.59 & 0.27 & 122 & 90.0 & 746.62 & 0.32 & 22 \\
			\hline
			- & - & $I_2$& 87.5 & 659.4 & 0.44 & 580 & 18.75 & 2556.11 & 0.32 & 411 & - & - & - & - \\
			\hline
			1 & 2 & $I_1$ & 100.0 & 62.17 & 0.42 & 274 & 90.91 & 436.87 & 0.28 & 130 & 63.64 & 1199.01 & 0.35 & 27 \\
			\hline
			- & - & $I_2$& 81.25 & 843.08 & 0.45 & 638 & 18.75 & 2594.5 & 0.32 & 420 & - & - & - & - \\
			\hline
			1 & 3 & $I_1$& 100.0 & 96.23 & 0.41 & 301 & 90.0 & 677.15 & 0.28 & 123 & 70.0 & 1241.38 & 0.35 & 27 \\
			\hline
			- & - & $I_2$& 82.35 & 909.16 & 0.45 & 694 & 17.65 & 2551.54 & 0.32 & 434 & - & - & - & - \\
			\hline
			2 & 1 & $I_1 \setminus \{1,4,7,9,10\}$& 100.0 & 40.48 & 0.39 & 216 & 83.33 & 760.57 & 0.27 & 96 & 66.67 & 1343.0 & 0.32 & 102 \\
			\hline
			- & - & $I_2 \setminus \{13,19,28\}$& 85.71 & 715.5 & 0.4 & 634 & 21.43 & 2482.98 & 0.3 & 296 & - & - & - & - \\
			\hline
			2 & 2 & $I_1 \setminus \{9,10\}$& 100.0 & 62.6 & 0.37 & 253 & 77.78 & 700.6 & 0.26 & 96 & 66.67 & 1102.47 & 0.36 & 97 \\
			\hline
			- & - & $I_2 \setminus \{19,28\}$& 80.0 & 1146.17 & 0.43 & 730 & 13.33 & 2623.81 & 0.32 & 290 & - & - & - & - \\
			\hline
			2 & 3 & $I_1 \setminus \{9\}$& 100.0 & 54.54 & 0.39 & 251 & 88.89 & 416.84 & 0.28 & 97 & 66.67 & 1488.62 & 0.4 & 116 \\
			\hline
			- & - & $I_2 \setminus \{28\}$& 68.75 & 1548.31 & 0.46 & 915 & 12.5 & 2705.04 & 0.35 & 292 & - & - & - & - \\
			\hline
			3 & 1 & 6,8 & 100.0 & 12.1 & 0.44 & 88 & 100.0 & 81.8 & 0.2 & 97 & 100.0 & 304.6 & 0.2 & 29 \\
			\hline
			- & - & 16,18,21,26& 75.0 & 841.42 & 0.37 & 516 & 50.0 & 1634.7 & 0.28 & 116 & - & - & - & - \\
			\hline
			3 & 2 & 2,3,6,8 & 100.0 & 16.97 & 0.46 & 170 & 100.0 & 183.33 & 0.26 & 102 & 66.67 & 1122.27 & 0.33 & 144 \\
			\hline
			- & - & 12,14,16,17,18,21,22,23,26& 100.0 & 462.28 & 0.44 & 576 & 12.5 & 2792.88 & 0.32 & 244 & - & - & - & - \\
			\hline
			3 & 3 & 2,3,5,6,8& 100.0 & 24.02 & 0.26 & 203 & 75.0 & 834.25 & 0.25 & 103 & 75.0 & 1363.18 & 0.4 & 172 \\
			\hline
			- & - & $I_2 \setminus \{15,19,20,24,25,28\}$& 90.0 & 877.32 & 0.49 & 764 & 20.0 & 2553.8 & 0.37 & 183 & - & - & - & - \\
			\hline                                                                                      
	\end{tabular}}
	\caption{Results on instances with non-uniform capacities and $k' \in \{1,2,3\}$} \label{tab:ResultsNonUnifKGEN}
\end{table}

Table \ref{tab:ResultsNonUnifKGEN}  shows the results in a similar way than in Table \ref{tab:ResultsNonUnifK0}, but in the case where the number $k'$ of arcs that can be protected is between 1 and 3. In this case, one can notice that the optimal value of the continuous relaxation of the cut-set formulation do not yields the same value that the flow  formulation one. However, on those instances, the optimal value of the continuous relaxation of the cut-set formulation yield values at least as good as the two other ones. However, as in the previous case, the bilevel formulation seems to be the most efficient one as the sizes of the instances grow. The results obtained with the three formulations also confirm that the solving time of an instance is sensitive to the value of $k'$: when this parameter grows, the solving time significantly increases.\\

\begin{figure}[h!]
	\centering
	\captionsetup{justification=centering}
	\begin{tikzpicture}
	\begin{axis}[
	width=0.6\textwidth,
	xlabel= Number of the instance,
	ylabel={Time (s)},
	xticklabels={},
	xmin=2, xmax=180,
	ymin=0, ymax=3105,
	legend pos=north west,
	ymajorgrids=true,
	grid style=dashed,
	]
	
	\addplot[
	color=black,
	densely dashed
	]
	coordinates { (3,1.3)  (4,2.5)  (5,2.3)  (6,3.5)  (7,4.2)  (8,4.7)  (9,4.3)  (10,1.5)  (11,2.3)  (12,5.6)  (13,6.1)  (14,3.3)  (15,4.2)  (16,18.5)  (17,31.1)  (18,4.2)  (19,18.9)  (20,3.7)  (21,5.7)  (22,4.3)  (23,19.4)  (24,27.9)  (25,17.4)  (26,20.1)  (27,21.4)  (28,14.4)  (29,7.7)  (30,29.1)  (31,47.4)  (32,3.4)  (33,5.5)  (34,66.1)  (35,11.1)  (36,8.7)  (37,37.9)  (38,45.2)  (39,32.8)  (40,13.8)  (41,95.8)  (42,6.6)  (43,15.1)  (44,8.6)  (45,39.7)  (46,12.1)  (47,138.2)  (48,10.2)  (49,41.9)  (50,45.5)  (51,98.3)  (52,12.2)  (53,192.4)  (54,49.0)  (55,22.2)  (56,15.0)  (57,11.3)  (58,14.8)  (59,66.3)  (60,14.6)  (61,49.3)  (62,96.3)  (63,76.2)  (64,24.7)  (65,242.2)  (66,100.9)  (67,85.8)  (68,29.7)  (69,102.6)  (70,47.0)  (71,249.0)  (72,66.9)  (73,98.2)  (74,87.5)  (75,61.3)  (76,89.5)  (77,29.8)  (78,104.7)  (79,274.8)  (80,172.2)  (81,135.4)  (82,536.4)  (83,147.2)  (84,207.4)  (85,4.8)  (86,17.8)  (87,38.2)  (88,274.2)  (89,348.5)  (90,131.1)  (91,579.3)  (92,130.8)  (93,63.1)  (94,126.6)  (95,229.7)  (96,152.4)  (97,354.0)  (98,622.2)  (99,102.5)  (100,315.9)  (101,333.4)  (102,660.4)  (103,975.2)  (104,252.3)  (105,566.6)  (106,562.0)  (107,720.9)  (108,1130.0)  (109,259.3)  (110,1084.0)  (111,1925.0)  (112,318.5)  (113,1357.0)  (114,196.4)  (115,300.8)  (116,229.9)  (117,347.4)  (118,1080.0)  (119,2041.0)  (120,180.3)  (121,352.0)  (122,1146.0)  (123,88.9)  (124,153.0)  (125,199.2)  (126,110.8)  (127,83.4)  (128,72.4)  (129,137.5)  (130,138.9)  (131,335.1)  (132,103.1)  (133,146.0)  (134,752.5)  (135,1366.0)  (136,2716.0)  (137,111.2)  (138,218.9)  (139,375.1)  (140,105.0)  (141,253.3)  (142,832.6)  (143,85.1)  (144,3000.0)  (145,3000.0)  (146,3000.0)  (147,3000.0)  (148,3000.0)  (149,3000.0)  (150,2206.0)  (151,3000.0)  (152,1961.0)  (153,3000.0)  (154,3000.0)  (155,625.4)  (156,3000.0)  (157,3000.0)  (158,333.6)  (159,778.0)  (160,1415.0)  (161,315.9)  (162,1170.0)  (163,3000.0)  (164,619.8)  (165,1086.0)  (166,1009.0)  (167,1234.0)  (168,2333.0)  (169,3000.0)  (170,160.7)  (171,569.3)  (172,1082.0)  (173,3000.0)  (174,3000.0)  (175,3000.0)  (176,3000.0)  (177,3000.0)  (178,3000.0)  (179,1253.0)  };
	\addlegendentry{BF}
	
	\addplot[
	color=blue,
	]
	coordinates {  (3,0.4)  (4,0.9)  (5,1.0)  (6,1.1)  (7,1.1)  (8,1.1)  (9,1.4)  (10,1.5)  (11,1.7)  (12,2.1)  (13,2.1)  (14,2.2)  (15,2.6)  (16,2.9)  (17,3.6)  (18,4.4)  (19,4.7)  (20,5.4)  (21,5.8)  (22,6.2)  (23,6.8)  (24,6.8)  (25,8.8)  (26,10.4)  (27,10.8)  (28,11.4)  (29,11.7)  (30,11.8)  (31,14.2)  (32,16.9)  (33,17.5)  (34,19.5)  (35,26.9)  (36,27.0)  (37,32.9)  (38,45.5)  (39,49.9)  (40,55.4)  (41,61.4)  (42,63.1)  (43,63.6)  (44,71.3)  (45,76.8)  (46,81.8)  (47,92.2)  (48,101.0)  (49,123.1)  (50,132.9)  (51,144.1)  (52,186.1)  (53,208.0)  (54,228.6)  (55,252.1)  (56,284.2)  (57,289.1)  (58,293.6)  (59,300.9)  (60,426.3)  (61,428.2)  (62,433.9)  (63,500.2)  (64,539.8)  (65,638.2)  (66,724.7)  (67,730.6)  (68,854.4)  (69,1317.0)  (70,1343.0)  (71,1372.0)  (72,1470.0)  (73,1539.0)  (74,1895.0)  (75,2179.0)  (76,2508.0)  (77,3000.0)  (78,3000.0)  (79,3000.0)  (80,3000.0)  (81,3000.0)  (82,3000.0)  (83,3000.0)  (84,3000.0)  (85,3000.0)  (86,3000.0)  (87,3000.0)  (88,3000.0)  (89,3000.0)  (90,3000.0)  (91,3000.0)  (92,3000.0)  (93,3000.0)  (94,3000.0)  (95,3000.0)  (96,3000.0)  (97,3000.0)  (98,3000.0)  (99,3000.0)  (100,3000.0)  (101,3000.0)  (102,3000.0)  (103,3000.0)  (104,3000.0)  (105,3000.0)  (106,3000.0)  (107,3000.0)  (108,3000.0)  (109,3000.0)  (110,3000.0)  (111,3000.0)  (112,3000.0)  (113,3000.0)  (114,3000.0)  (115,3000.0)  (116,3000.0)  (117,3000.0)  (118,3000.0)  (119,3000.0)  (120,3000.0)  (121,3000.0)  (122,3000.0)  (123,3000.0)  (124,3000.0)  (125,3000.0)  (126,3000.0)  (127,3000.0)  (128,3000.0)  (129,3000.0)  (130,3000.0)  (131,3000.0)  (132,3000.0)  (133,3000.0)  (134,3000.0)  (135,3000.0)  (136,3000.0)  (137,3000.0)  (138,3000.0)  (139,3000.0)  (140,3000.0)  (141,3000.0)  (142,3000.0)  (143,3000.0)  (144,3000.0)  (145,3000.0)  (146,3000.0)  (147,3000.0)  (148,3000.0)  (149,3000.0)  (150,3000.0)  (151,3000.0)  (152,3000.0)  (153,3000.0)  (154,3000.0)  (155,3000.0)  (156,3000.0)  (157,3000.0)  (158,3000.0)  (159,3000.0)  (160,3000.0)  (161,3000.0)  (162,3000.0)  (163,3000.0)  (164,3000.0)  (165,3000.0)  (166,3000.0)  (167,3000.0)  (168,3000.0)  (169,3000.0)  (170,3000.0)  (171,3000.0)  (172,3000.0)  (173,3000.0)  (174,3000.0)  (175,3000.0)  (176,3000.0)  (177,3000.0)  (178,3000.0)  (179,3000.0)  };
	\addlegendentry{CF}
	
	\addplot[
	color=red,
	dotted
	]
	coordinates { (3,0.1)  (4,0.1)  (5,1.2)  (6,0.1)  (7,1.1)  (8,1.0)  (9,1.0)  (10,0.3)  (11,1.0)  (12,1.0)  (13,1.0)  (14,1.4)  (15,1.0)  (16,966.2)  (17,70.2)  (18,6.2)  (19,6.7)  (20,2.8)  (21,7.5)  (22,5.2)  (23,22.7)  (24,1831.0)  (25,313.4)  (26,14.7)  (28,19.7)  (29,19.2)  (30,199.3)  (31,1325.0)  (32,13.2)  (33,6.6)  (34,164.6)  (35,10.4)  (36,8.3)  (37,3000.0)  (38,3000.0)  (39,1243.0)  (40,218.0)  (41,295.3)  (42,8.5)   (44,208.6)  (45,560.0)  (46,304.6)  (47,958.8)  (48,77.0)  (49,315.1)  (50,3000.0)  (51,2531.0)  (52,309.3)  (53,1192.0)  (54,2760.0)  (55,3000.0)  (56,243.5)  (57,53.4)  (59,3000.0)  (61,853.7)  (63,3000.0)  (65,3000.0)  (66,3000.0)  (67,3000.0)   (72,3000.0)   (75,2422.0)  (77,3000.0)  (78,3000.0)  (79,3000.0)  (80,3000.0)  (81,3000.0)  (82,3000.0)  (83,3000.0)  (84,3000.0)  };
	
	\addlegendentry{FF}

	\end{axis}
	\end{tikzpicture}
	\centering \caption{Solving time of the different instances regarding the three formulations} \label{fig:solvingTime}
\end{figure}
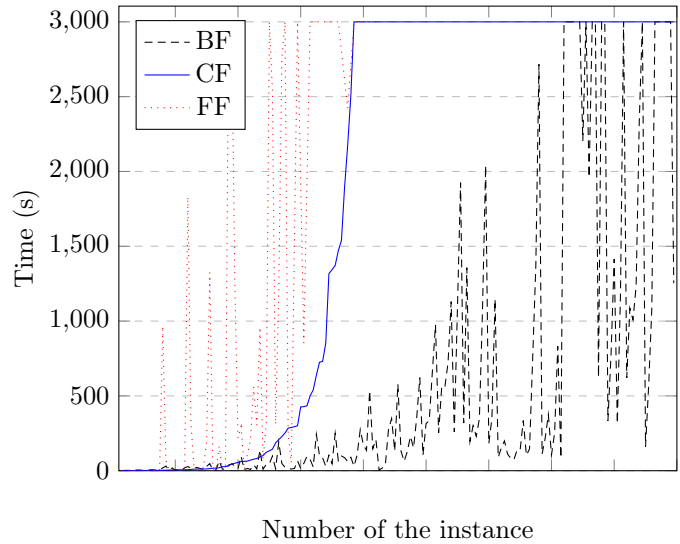

\pgfplotstableread{
	1 0.32083333333333336 0.0 0.47187499999999993 0.0
	2 0.20915712799167532 0.1519250780437044 0.42873176206509545 0.27946127946127947
	3 0.40813464235624125 0.2531556802244039 0.5291970802919708 0.35474452554744523
	4 0.4149808638600328 0.2772006560962274 0.48311390955924444 0.363480251860332
	5 0.2147774208352455 0.17806333180357958 0.4364326375711574 0.3932637571157495
	6 0.23024830699774257 0.0 0.26781326781326786 0.11302211302211311
	7 0.3982300884955752 0.18905872888173772 0.5114573785517873 0.19706691109074248
	8 0.20252277505255772 0.15977575332866148 0.36604361370716504 0.2982866043613707
	9 0.3074285714285715 0.2365714285714286 0.4443091905051734 0.3432744978697505
	10 0.34040047114252064 0.0977620730270907 0.44354838709677424 0.13844086021505378
	11 0.2706164931945557 0.20256204963971172 0.376271186440678 0.30762711864406783
	12 0.26280487804878044 0.21341463414634138 0.49307844429795644 0.4027686222808174
	13 0.20387409200968512 0.16513317191283294 0.403 0.36849999999999994
	14 0.16412859560067677 0.07106598984771573 0.36346516007532953 0.22033898305084745
	15 0.3957894736842106 0.2821052631578947 0.5253595760787283 0.38077214231642703
	16 0.3434473854099419 0.24919302775984503 0.4485094850948509 0.3523035230352303
	17 0.14285714285714277 0.09854014598540149 0.4194798007747648 0.3486441615938019
	18 0.33420822397200345 0.2230971128608923 0.4190207156308851 0.2853107344632768
	19 0.31301939058171746 0.17174515235457055 0.4599427753934192 0.33261802575107297
	20 0.3977955911823648 0.2600200400801604 0.5138821237213833 0.38723818801753535
	21 0.22520420070011674 0.15227537922987172 0.435912938331318 0.375453446191052
	22 0.10611246943765272 0.07775061124694375 0.32640648758236196 0.2919412062848454
	23 0.27146311970979436 0.1904474002418379 0.42793650793650795 0.34476190476190477
	24 0.3015297906602254 0.2403381642512077 0.4800619434765776 0.3987611304684475
}\datatable

\makeatletter
\begin{figure}[h!]
	\begin{tikzpicture}
	
	\begin{axis}[
	width=0.7\textwidth,
	ylabel= Gap,
	xlabel= Number of the instance,
	xtick=data,
	xticklabels={1,2,3,4,5,6,7,8,9,11,13,14,15,16,17,18,20,21,22,23,24,25,28,29,30},
	ymin=0,ymax=0.56,
	legend style={
		font=\footnotesize,
		cells={anchor=west},
		legend columns=5,
		at={(0.65,0.11)},
		anchor=north,
		xticklabel style = {font=\tiny,yshift=0.5ex}
	},
	]
	\addplot[color=red, mark=square, dashed] table[x index=0,y index=1] \datatable;
	\addplot[color=black,mark=o, dotted] table[x index=0,y index=2] \datatable;
	\addplot[color=brown,mark=square] table[x index=0,y index=3] \datatable;
	\addplot[color=blue,mark=o] table[x index=0,y index=4] \datatable;
	\legend{$gap_B^0$,$gap_C^0$,$gap_B^1$,$gap_C^1$}
	\end{axis}
	\end{tikzpicture}
	\centering \caption{ Gap between the optimal values of the bilevel and cut-set continuous relaxations and the optimal value of the problem for the different instances when $k=1$ and $k' \in\{0,1\}$} \label{fig:gap}
\end{figure}
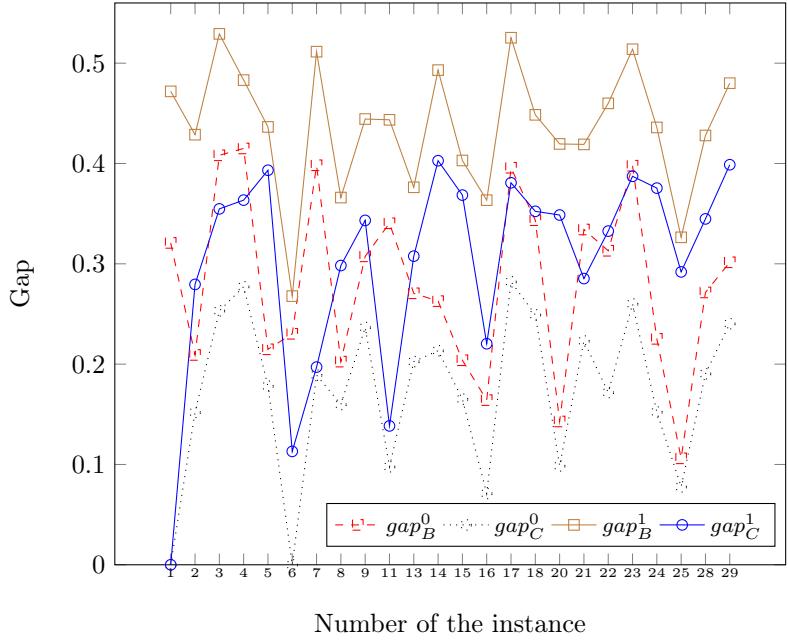

\begin{table}[h!]
	\begin{subtable}[t]{1\linewidth}
		\centering
		\scalebox{0.71}{\begin{tabular}{|r|r|r|r?r|r|r?r|r|r?r|r|}
				\hline
				\multicolumn{4}{|c?}{Parameters} & \multicolumn{3}{c?}{Bilevel} & \multicolumn{3}{c?}{Cut-set} & \multicolumn{2}{c|}{Flow} \\
				\hline
				I & U & k & opt & CR & time (s) & it & CR & time (s) & it & time (s) & it \\
				\Xhline{1.5pt}
				2  &  2  &  1  &  8.63  &  \cellcolor{lightgray} 7.47  &  0.4  &  12  &  7.41  &  0.2  &  26  &  \cellcolor{lightgray} 0.1  &  2\\
				\hline
				-  &  2  &  2  &  14.87  &  \cellcolor{lightgray} 12.29  &  1.3  &  27  &  11.95  &  \cellcolor{lightgray} 0.6  &  52  &  1.2  &  6\\
				\hline
				-  &  3  &  1  &  7.99  &  6.97  &  0.2  &  6  &  \cellcolor{lightgray} 6.99  &  0.1  &  9  &  \cellcolor{lightgray} 0.1  &  2\\
				\hline
				-  &  3  &  2  &  13.17  &  11.41  &  1.0  &  21  &  \cellcolor{lightgray} 11.41  &  \cellcolor{lightgray} 0.2  &  13  &  2.0  &  11\\
				\hline
				3  &  4  &  1  &  9.89  &  \cellcolor{lightgray} 6.65  &  0.9  &  24  &  6.13  &  \cellcolor{lightgray} 0.2  &  20  &  0.7  &  19\\
				\hline
				-  &  4  &  2  &  16.5  &  \cellcolor{lightgray} 13.62  &  2.0  &  44  &  13.52  &  \cellcolor{lightgray} 0.1  &  6  &  8.6  &  34\\
				\hline
				-  &  5  &  1  &  9.89  &  \cellcolor{lightgray} 6.32  &  0.9  &  23  &  5.86  &  \cellcolor{lightgray} 0.2  &  18  &  1.2  &  17\\
				\hline
				-  &  5  &  2  &  16.5  &  \cellcolor{lightgray} 13.44  &  2.8  &  55  &  13.19  &  \cellcolor{lightgray} 0.1  &  6  &  14.5  &  41\\ 
				\hline                                                                                        
				8  &  5  &  1  &  12.76  &  11.05  &  1.1  &  22  &  \cellcolor{lightgray} 11.18  &  \cellcolor{lightgray} 0.4  &  24  &  1.1  &  11\\
				\hline
				-  &  5  &  2  &  20.5  &  17.88  &  5.3  &  75  &  \cellcolor{lightgray} 18.21  &  \cellcolor{lightgray} 0.5  &  30  &  76.6  &  42\\
				\hline
				-  &  5  &  3  &  30.56  &  28.14  &  4.8  &  48  &  \cellcolor{lightgray} 28.32  &  \cellcolor{lightgray} 0.3  &  14  &  771.6  &  59\\
				\hline
				-  &  6  &  1  &  12.67  &  10.87  &  1.1  &  22  &  \cellcolor{lightgray} 10.94  &  \cellcolor{lightgray} 0.4  &  23  &  1.7  &  9\\
				\hline
				-  &  6  &  2  &  20.5  &  17.86  &  8.3  &  110  &  \cellcolor{lightgray} 17.97  &  \cellcolor{lightgray} 0.5  &  32  &  137.9  &  64\\
				\hline
				-  &  6  &  3  &  30.56  &  27.9  &  10.0  &  90  &  \cellcolor{lightgray} 28.04  &  \cellcolor{lightgray} 0.2  &  12  &  1548  &  50\\ 
				\hline
		\end{tabular}}
	\end{subtable}%
	\vspace{5pt}
	\begin{subtable}[t]{0.55\linewidth}
		\centering
		\scalebox{0.7}{\begin{tabular}{|r|r|r|r?r|r|r?r|r|r|}
				\hline
				\multicolumn{4}{|c?}{Parameters} & \multicolumn{3}{c?}{Bilevel} & \multicolumn{3}{c|}{Cut-set} \\
				\hline
				I & U & k & opt & CR & time (s) & it & CR & time (s) & it \\
				\Xhline{1.5pt}
				22  &  5  &  1  &  11.8  &  \cellcolor{lightgray} 9.44  &  \cellcolor{lightgray} 3.6  &  30  &  9.12  &  19.2  &  470  \\
				\hline
				-  &  5  &  2  &  19.56  &  15.88  &  \cellcolor{lightgray} 12.7  &  67  &  \cellcolor{lightgray} 16.0  &  16.2  &  356  \\
				\hline
				-  &  6  &  1  &  11.31  &  \cellcolor{lightgray} 9.2  &  \cellcolor{lightgray} 3.8  &  38  &  8.92  &  7.2  &  191  \\
				\hline
				-  &  6  &  2  &  19.56  &  15.71  &  24.8  &  125  &  \cellcolor{lightgray} 15.8  &  \cellcolor{lightgray} 13.8  &  301  \\
				\hline
				23  &  7  &  1  &  13.43  &  \cellcolor{lightgray} 11.44  &  11.8  &  82  &  11.19  &  \cellcolor{lightgray} 7.1  &  198  \\
				\hline
				-  &  7  &  2  &  22.45  &  18.79  &  53.9  &  267  &  \cellcolor{lightgray} 19.47  &  \cellcolor{lightgray} 12.6  &  254  \\
				\hline
				-  &  10  &  1  &  12.85  &  \cellcolor{lightgray} 10.37  &  6.6  &  48  &  10.33  &  \cellcolor{lightgray} 6.0  &  169  \\
				\hline
				-  &  10  &  2  &  22.45  &  18.39  &  36.4  &  195  &  \cellcolor{lightgray} 18.6  &  \cellcolor{lightgray} 11.6  &  211  \\
				\hline
				24  &  14  &  1  &  14.89  &  \cellcolor{lightgray} 12.91  &  15.6  &  97  &  12.64  &  \cellcolor{lightgray} 4.1  &  90  \\
				\hline
				-  &  14  &  2  &  27.19  &  25.73  &  15.4  &  64  &  \cellcolor{lightgray} 25.85  &  \cellcolor{lightgray} 2.2  &  39  \\
				\hline
				-  &  19  &  1  &  14.36  &  \cellcolor{lightgray} 12.55  &  5.3  &  35  &  12.5  &  \cellcolor{lightgray} 2.8  &  68  \\
				\hline
				-  &  19  &  2  &  27.15  &  25.62  &  7.9  &  34  &  \cellcolor{lightgray} 25.68  &  \cellcolor{lightgray} 1.3  &  23  \\
				\hline                                                                          
				29  &  2  &  1  &  6.96  &  6.25  &  3.9  &  31  &  \cellcolor{lightgray} 6.34  &  \cellcolor{lightgray} 2.1  &  41  \\
				\hline
				-  &  2  &  2  &  12.08  &  \cellcolor{lightgray} 10.19  &  15.1  &  103  &  10.18  &  \cellcolor{lightgray} 4.4  &  88  \\
				\hline
				-  &  2  &  3  &  17.45  &  \cellcolor{lightgray} 15.66  &  12.0  &  69  &  15.57  &  \cellcolor{lightgray} 4.9  &  105  \\
				\hline
				-  &  3  &  1  &  6.62  &  6.05  &  2.9  &  22  &  \cellcolor{lightgray} 6.12  &  \cellcolor{lightgray} 0.1  &  2  \\
				\hline
				-  &  3  &  2  &  12.08  &  \cellcolor{lightgray} 9.93  &  20.7  &  116  &  9.9  &  \cellcolor{lightgray} 4.9  &  99  \\
				\hline
				-  &  3  &  3  &  17.45  &  \cellcolor{lightgray} 15.42  &  11.9  &  66  &  15.29  &  \cellcolor{lightgray} 2.9  &  54  \\
				\hline
		\end{tabular}}
	\end{subtable}%
	\begin{subtable}[t]{0.55\linewidth}
		\centering
		\scalebox{0.7}{\begin{tabular}{|r|r|r|r?r|r|r?r|r|r|}
				\hline
				\multicolumn{4}{|c?}{Parameters} & \multicolumn{3}{c?}{Bilevel} & \multicolumn{3}{c|}{Cut-set} \\
				\hline
				I & U & k & opt & CR & time (s) & it & CR & time (s) & it \\
				\Xhline{1.5pt}                                                                                        30&  3  &  1  &  8.82  &  6.08  &  12.6  &  70  & \cellcolor{lightgray} 5.81  & \cellcolor{lightgray}  11.5  &  1220  \\
				\hline
				-  &  3  &  2  &  15.18  &  \cellcolor{lightgray} 10.01  &  100.0  &  258  &  9.97  &  \cellcolor{lightgray} 88.7  &  739  \\
				\hline
				-  &  4  &  1  &  8.82  &  5.49  &  \cellcolor{lightgray} 17.6  &  56  &  \cellcolor{lightgray} 5.49  &  22.4  &  331  \\
				\hline
				-  &  4  &  2  &  15.18  &  \cellcolor{lightgray} 9.67  &  \cellcolor{lightgray} 114.9  &  243  &  9.51  &  166.4  &  799  \\
				\hline
				31  &  6  &  1  &  13.37  &  \cellcolor{lightgray} 11.0  &  \cellcolor{lightgray} 60.1  &  204  &  10.71  &  130.8  &  1031  \\
				\hline
				-  &  6  &  2  &  21.6  &  \cellcolor{lightgray} 18.72  &  120.4  &  322  &  18.58  &  \cellcolor{lightgray} 14.9  &  195  \\
				\hline
				-  &  8  &  1  &  12.38  &  \cellcolor{lightgray} 10.6  &  19.9  &  94  &  10.48  &  \cellcolor{lightgray} 13.5  &  206  \\
				\hline
				-  &  8  &  2  &  21.41  &  \cellcolor{lightgray} 18.45  &  107.9  &  306  &  18.28  &  \cellcolor{lightgray} 21.2  &  274  \\
				\hline
				32  &  9  &  1  &  11.55  &  8.36  &  \cellcolor{lightgray} 11.7  &  56  &  \cellcolor{lightgray} 8.36  &  16.5  &  231  \\
				\hline
				-  &  9  &  2  &  21.6  &  \cellcolor{lightgray} 16.66  &  126.5  &  264  &  16.62  &  \cellcolor{lightgray} 23.0  &  254  \\
				\hline
				-  &  12  &  1  &  11.5  &  \cellcolor{lightgray} 8.25  &  12.4  &  59  &  8.19  &  \cellcolor{lightgray} 9.5  &  139  \\
				\hline
				-  &  12  &  2  &  21.34  &  \cellcolor{lightgray} 16.52  &  79.8  &  194  &  16.43  &  \cellcolor{lightgray} 14.2  &  158  \\
				\hline
				33  &  18  &  1  &  15.21  &  14.83  &  5.8  &  25  &  \cellcolor{lightgray} 14.84  &  \cellcolor{lightgray} 0.6  &  8  \\
				\hline
				-  &  18  &  2  &  26.64  &  26.02  &  4.7  &  14  &  \cellcolor{lightgray} 26.15  &  \cellcolor{lightgray} 0.8  &  9  \\
				\hline
				-  &  24  &  1  &  15.21  &  14.76  &  5.1  &  24  &  \cellcolor{lightgray} 14.76  &  \cellcolor{lightgray} 0.7  &  8  \\
				\hline
				-  &  24  &  2  &  26.64  &  25.93  &  4.5  &  13  &  \cellcolor{lightgray} 26.01  &  \cellcolor{lightgray} 0.8  &  9  \\
				\hline
				34  &  29  &  1  &  19.02  &  \cellcolor{lightgray} 17.64  &  12.0  &  41  &  17.53  &  \cellcolor{lightgray} 7.7  &  82  \\
				\hline
				-  &  39  &  1  &  19.02  &  \cellcolor{lightgray} 17.33  &  8.9  &  30  &  17.28  &  \cellcolor{lightgray} 4.7  &  55  \\
				\hline
		\end{tabular}}
	\end{subtable}
	\caption{Results on instances with uniform capacities and $k'=0$} \label{tab:ResultsUniK0}
\end{table}

In Figure \ref{fig:solvingTime}, we can see the solving time for the bilevel, the cut-set and the flow formulations (respectively BF, CF, and FF in the legend). Figures \ref{fig:solvingTime} and \ref{fig:gap} present the results obtained on instances numbered from 1 to 29. The solving time is still bounded by 3000 seconds (when the solving time is equal to 3000 in this figure, it means that the instance has not been solved at the end of the allocated time). It shows that the bilevel formulation is more resilient to the growth of the instances size, whereas the flow formulation seems to be the more sensitive one. \\

Figure \ref{fig:gap} shows the evolution of the gap between the optimal values of the different continuous relaxations and the optimal value of the problem on each instance, for $k=1$ and $k' \in \{0,1\}$. In the legend, $gap_B^0$ (respectively $gap_C^0$) corresponds to the gap (in percentage) between the optimal value of the continuous relaxation of the bilevel (respectively cut-set) formulation and the optimal value of the problem when $k'=0$, whereas $gap_B^1$ (respectively $gap_C^1$) deals with the case where $k'=1$. First, we can see that the gap for the cut-set formulation is better than the one for the bilevel formulation in both cases. Second, one can notice that the addition of protected arcs greatly deteriorates the optimal value of the continuous relaxation in both cases, as the gap increases significantly. This could explain why the different formulations are sensitive to the value of $k'$.\\

Table \ref{tab:ResultsUniK0} gives the results for the three formulations in the case of uniform capacities. Each instance presented in Table \ref{tab:parametersSteiner} has been tested with two different uniform capacities (as previously, those capacities have been chosen according to the number of terminals, and are equal to $0.8|T|$ and $0.6|T|$). According to these results, the reformulation of the cut-set formulation when capacities are uniform is particularly interesting, as this formulation does seem to be the best one for this particular case. Furthermore, the solving time is highly reduced in comparison with the non-uniform case. One can also notice that, in this case, the best continuous relaxation is not always obtained with the cut-set formulation, as the bilevel formulation often produces one of higher quality.\\

 We also compute the results obtained by the three formulations, without the addition of the valid inequalities proposed in \ref{subsec:validIneq}, on a subset of instances with $k \in \{1,2,3\}$, $k'=0$, and non-uniform capacities. Let $\overline \Delta_{time}^B$ (respectively $\overline \Delta_{time}^C$ and $\overline \Delta_{time}^F$) be the mean augmentation of the solving time when these valid inequalities are removed from the bilevel (respectively cut-set and flow) formulation. On the test instances, $\overline \Delta_{time}^B$ is equal to 3.24 (meaning that the solving time is multiplied by 3.24 on average without the valid inequalities), while $\overline \Delta_{time}^C$ and $\overline \Delta_{time}^F$ are equal to 46.78 and 642.94 respectively. Hence, adding these valid inequalities has a huge impact on the solving time, especially on the flow and cut-set formulations. Furthermore, let $\overline \Delta_{CR}^B$ (respectively $\overline \Delta_{CR}^C$) be the mean augmentation of the optimal value of the continuous relaxation when these valid inequalities are added to the bilevel (respectively cut-set) formulation. On the test instances, $\overline \Delta_{CR}^B$ is equal to 1.48 (meaning that the optimal value of the continuous relaxation for the bilevel formulation is multiplied by 1.48 on average with these valid inequalities), while $\overline \Delta_{CR}^C$ is equal to 1.28 (in this case, when $k'=0$, the continuous relaxations of the cut-set and flow formulations yield the same optimal values on this set of instances). The optimal value of the continuous relaxation is then consequently increased when we add these valid inequalities, especially with the bilevel formulation. \\

\pgfplotsset{
    draw group line/.style n args={5}{
        after end axis/.append code={
            \setcounter{groupcount}{0}
            \pgfplotstableforeachcolumnelement{#1}\of\datatable\as\cell{%
                \def\temp{#2}
                \ifx\temp\cell
                    \ifnum\thegroupcount=0
                        \stepcounter{groupcount}
                        \pgfplotstablegetelem{\pgfplotstablerow}{[index]0}\of\datatable
                        \coordinate [yshift=#4] (startgroup) at (axis cs:\pgfplotsretval,0);
                    \else
                        \pgfplotstablegetelem{\pgfplotstablerow}{[index]0}\of\datatable
                        \coordinate [yshift=#4] (endgroup) at (axis cs:\pgfplotsretval,0);
                    \fi
                \else
                    \ifnum\thegroupcount=1
                        \setcounter{groupcount}{0}
                        \draw [
                            shorten >=-#5,
                            shorten <=-#5
                        ] (startgroup) -- node [anchor=north] {#3} (endgroup);
                    \fi
                \fi
            }
            \ifnum\thegroupcount=1
                        \setcounter{groupcount}{0}
                        \draw [
                            shorten >=-#5,
                            shorten <=-#5
                        ] (startgroup) -- node [anchor=north] {#3} (endgroup);
            \fi
        }
    }
}
\makeatother

\pgfplotsset{select coords between index/.style 2 args={
		x filter/.code={
			\ifnum\coordindex<#1\def\pgfmathresult{}\fi
			\ifnum\coordindex>#2\def\pgfmathresult{}\fi
		}
}}

 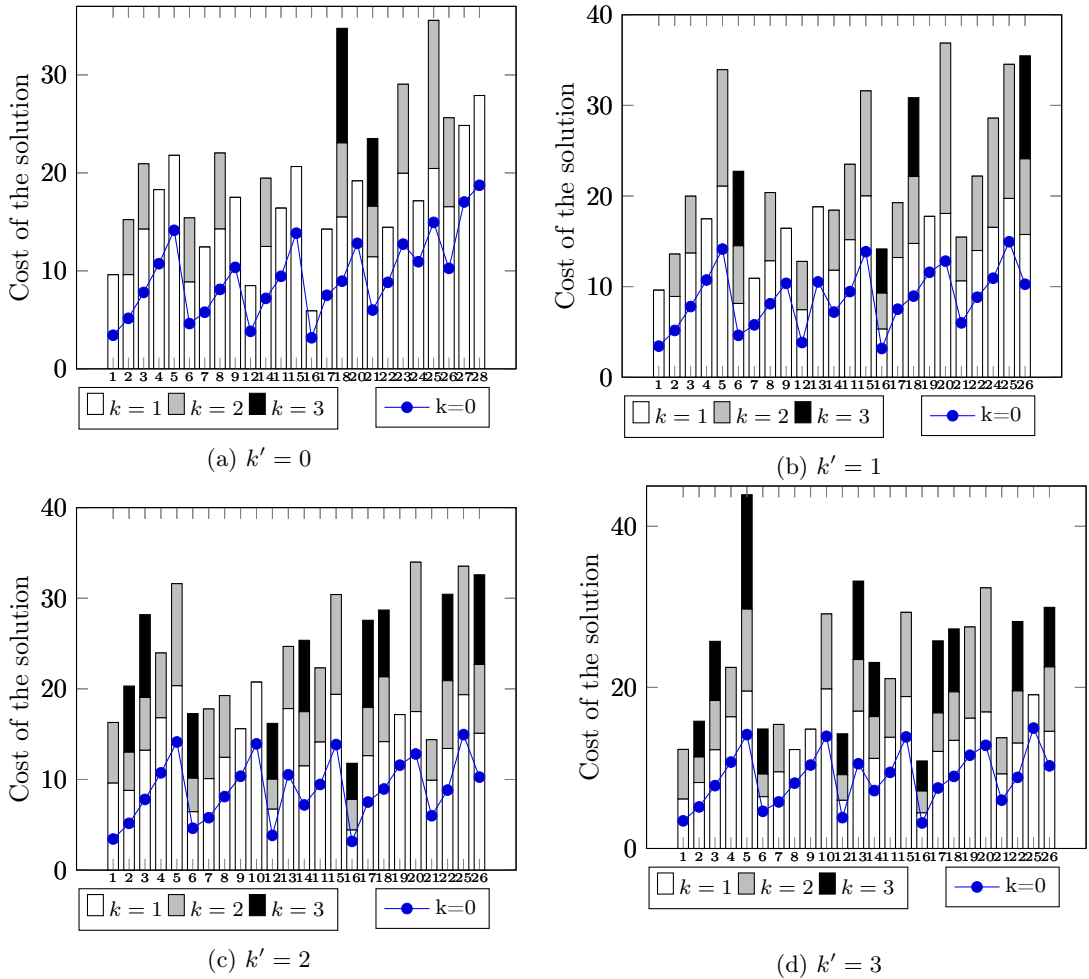
\begin{figure}[h!]
	\begin{subfigure}{0.5\linewidth}
		\begin{center}
\begin{tikzpicture}
\pgfplotstableread{
	1 9.6 0.0 0.0 3.43     
	2 9.61 5.620000000000001 0.0 5.16    
	3 14.26 6.67 0.0 7.8   
	4 18.29 0.0 0.0 10.73    
	5 21.79 0.0 0.0 14.14    
	6 8.86 6.550000000000001 0.0 4.62   
	7 12.43 0.0 0.0 5.78     
	8 14.27 7.760000000000002 0.0 8.11    
	9 17.5 0.0 0.0 10.36    
	10 8.49 0.0 0.0 3.83   
	11 12.49 6.970000000000001 0.0 7.19    
	12 16.4 0.0 0.0 9.45   
	13 20.65 0.0 0.0 13.85    
	14 5.91 0.0 0.0 3.17    
	15 14.25 0.0 0.0 7.51    
	16 15.49 7.5699999999999985 11.7 8.95   
	17 19.18 0.0 0.0 12.81   
	18 11.43 5.16 6.920000000000002 6.0    
	19 14.44 0.0 0.0 8.83    
	20 19.96 9.09 0.0 12.73    
	21 17.14 0.0 0.0 10.93   
	22 20.45 15.120000000000001 0.0 14.95    
	23 16.54 9.09 0.0 10.26     
	24 24.84 0.0 0.0 17.03 
	25 27.9 0.0 0.0 18.74
}\datatable

\begin{axis}[
		width=0.99\textwidth,
ylabel=Cost of the solution,
y label style={at={(axis description cs:0.08,.5)}},
xtick=data,
xticklabels={1,2,3,4,5,6,7,8,9,12,14,11,15,16,17,18,20,21,22,23,24,25,26,27,28},
ymin=0,ymax=37,
ybar stacked,
bar width=4pt,
legend style={
	font=\footnotesize,
	cells={anchor=west},
	legend columns=5,
	at={(0.3,-0.05)},
	anchor=north,
	xticklabel style = {font=\tiny,yshift=0.5ex}
},
]

\addplot[fill=white] table[x index=0,y index=1] \datatable;
\addplot[fill=lightgray] table[x index=0,y index=2] \datatable;
\addplot[fill=black] table[x index=0,y index=3] \datatable;
\legend{$k=1$,$k=2$,$k=3$}
\end{axis}

\begin{axis}[
		width=0.99\textwidth,
xtick=data,
xticklabels={1,2,3,4,5,6,7,8,9,12,14,11,15,16,17,18,20,21,22,23,24,25,26,27,28},
ymin=0,ymax=37,
legend style={
	font=\footnotesize,
	cells={anchor=west},
	legend columns=5,
	at={(0.81,-0.05)},
	anchor=north,
	xticklabel style = {font=\tiny,yshift=0.5ex}
},
]
\addplot table[x index=0,y index=4] \datatable;
\legend{k=0}
\end{axis}
\end{tikzpicture}
			\caption{$k'=0$}
		\end{center}
	\end{subfigure}
	\hfill
	\begin{subfigure}{0.5\linewidth}
		\pgfplotstableread{
			1 9.6 0.0 0.0 3.43  
			2 8.91 4.68 0.0 5.16  
			3 13.7 6.280000000000001 0.0 7.8  
			4 17.47 0.0 0.0 10.73  
			5 21.08 12.86 0.0 14.14  
			6 8.14 6.359999999999999 8.23 4.62  
			7 10.91 0.0 0.0 5.78  
			8 12.84 7.539999999999999 0.0 8.11  
			9 16.43 0.0 0.0 10.36  
			10 7.44 5.339999999999999 0.0 3.83  
			11 18.79 0.0 0.0 10.52  
			12 11.8 6.629999999999999 0.0 7.19  
			13 15.17 8.33 0.0 9.45  
			14 20.0 11.600000000000001 0.0 13.85  
			15 5.31 3.9699999999999998 4.880000000000001 3.17  
			16 13.21 6.050000000000001 0.0 7.51  
			17 14.76 7.389999999999999 8.700000000000003 8.95  
			18 17.75 0.0 0.0 11.58  
			19 18.07 18.799999999999997 0.0 12.81  
			20 10.62 4.840000000000002 0.0 6.0  
			21 13.98 8.21 0.0 8.83  
			22 16.54 12.05 0.0 10.93  
			23 19.73 14.8 0.0 14.95  
			24 15.75 8.350000000000001 11.369999999999997 10.26
		}\datatable

		\makeatletter
		\begin{tikzpicture}
		\begin{axis}[
		width=0.99\textwidth,
		ylabel=Cost of the solution,
		y label style={at={(axis description cs:0.08,.5)}},
		xtick=data,
		xticklabels={1,2,3,4,5,6,7,8,9,12,13,14,11,15,16,17,18,19,20,21,22,24,25,26},
		ymin=0,ymax=40,
		ybar stacked,
		bar width=4pt,
		legend style={
			font=\footnotesize,
			cells={anchor=west},
			legend columns=5,
			at={(0.3,-0.05)},
			anchor=north,
			xticklabel style = {font=\tiny,yshift=0.5ex}
		},
		]
		\addplot[fill=white] table[x index=0,y index=1] \datatable;
		\addplot[fill=lightgray] table[x index=0,y index=2] \datatable;
		\addplot[fill=black] table[x index=0,y index=3] \datatable;
		\legend{$k=1$,$k=2$,$k=3$}
		\end{axis}

		\begin{axis}[
		width=0.99\textwidth,
		xtick=data,
		xticklabels={1,2,3,4,5,6,7,8,9,12,13,14,11,15,16,17,18,19,20,21,22,24,25,26},
		ymin=0,ymax=40,
		legend style={
			font=\footnotesize,
			cells={anchor=west},
			legend columns=5,
			at={(0.81,-0.05)},
			anchor=north,
			xticklabel style = {font=\tiny,yshift=0.5ex}
		},
		]
		\addplot table[x index=0,y index=4] \datatable;
		\legend{k=0}
		\end{axis}
		\end{tikzpicture}
					\caption{$k'=1$}
	\end{subfigure}
	\begin{subfigure}{0.5\linewidth}
		\begin{center}
			\pgfplotstableread{
				1 9.6 6.6899999999999995 0.0 3.43 
				2 8.79 4.220000000000001 7.299999999999999 5.16 
				3 13.22 5.839999999999998 9.130000000000003 7.8 
				4 16.8 7.16 0.0 10.73 
				5 20.34 11.27 0.0 14.14 
				6 6.44 3.7 7.129999999999999 4.62 
				7 10.09 7.690000000000001 0.0 5.78 
				8 12.44 6.820000000000002 0.0 8.11 
				9 15.59 0.0 0.0 10.36 
				10 20.74 0.0 0.0 13.93 
				11 6.72 3.3099999999999996 6.160000000000002 3.83 
				12 17.82 6.859999999999999 0.0 10.52 
				13 11.49 5.99 7.879999999999999 7.19 
				14 14.13 8.179999999999998 0.0 9.45 
				15 19.39 11.02 0.0 13.85 
				16 4.42 3.37 3.999999999999999 3.17 
				17 12.62 5.340000000000002 9.61 7.51 
				18 14.17 7.15 7.379999999999999 8.95 
				19 17.16 0.0 0.0 11.58 
				20 17.49 16.49 0.0 12.81 
				21 9.91 4.48 0.0 6.0 
				22 13.42 7.479999999999999 9.540000000000003 8.83 
				23 19.35 14.18 0.0 14.95 
				24 15.09 7.600000000000001 9.900000000000002 10.26 
			}\datatable

			\makeatletter
			\begin{tikzpicture}
			\begin{axis}[
			width=0.99\textwidth,
			ylabel=Cost of the solution,
			y label style={at={(axis description cs:0.08,.5)}},
			xtick=data,
			xticklabels={1,2,3,4,5,6,7,8,9,10,12,13,14,11,15,16,17,18,19,20,21,22,25,26},
			ymin=0,ymax=40,
			ybar stacked,
			bar width=4pt,
			legend style={
				font=\footnotesize,
				cells={anchor=west},
				legend columns=5,
				at={(0.3,-0.05)},
				anchor=north,
				xticklabel style = {font=\tiny,yshift=0.5ex}
			},
			]
			
			\addplot[fill=white] table[x index=0,y index=1] \datatable;
			\addplot[fill=lightgray] table[x index=0,y index=2] \datatable;
			\addplot[fill=black] table[x index=0,y index=3] \datatable;
			\legend{$k=1$,$k=2$,$k=3$}
			\end{axis}

			\begin{axis}[
			width=0.99\textwidth,
			xtick=data,
			xticklabels={1,2,3,4,5,6,7,8,9,10,12,13,14,11,15,16,17,18,19,20,21,22,25,26},
			ymin=0,ymax=40,
			legend style={
				font=\footnotesize,
				cells={anchor=west},
				legend columns=5,
				at={(0.81,-0.05)},
				anchor=north,
				xticklabel style = {font=\tiny,yshift=0.5ex}
			},
			]
			\addplot table[x index=0,y index=4] \datatable;
			\legend{k=0}
			\end{axis}
			\end{tikzpicture}
			\caption{$k'=2$}
		\end{center}
	\end{subfigure}
	\begin{subfigure}{0.5\linewidth}
		\begin{center}
			\pgfplotstableread{
				1 6.13 6.170000000000001 0.0 3.43 
				2 8.19 3.17 4.440000000000001 5.16 
				3 12.25 6.120000000000001 7.349999999999998 7.8 
				4 16.34 6.140000000000001 0.0 10.73 
				5 19.53 10.189999999999998 14.200000000000003 14.14 
				6 6.43 2.8100000000000005 5.6 4.62 
				7 9.51 5.890000000000001 0.0 5.78 
				8 12.28 0.0 0.0 8.11 
				9 14.8 0.0 0.0 10.36 
				10 19.8 9.32 0.0 13.93 
				11 5.98 3.1799999999999997 5.08 3.83 
				12 17.04 6.420000000000002 9.740000000000002 10.52 
				13 11.18 5.190000000000001 6.73 7.19 
				14 13.8 7.279999999999998 0.0 9.45 
				15 18.86 10.46 0.0 13.85 
				16 4.42 2.71 3.7399999999999993 3.17 
				17 12.05 4.779999999999998 8.950000000000003 7.51 
				18 13.42 6.020000000000001 7.82 8.95 
				19 16.17 11.329999999999998 0.0 11.58 
				20 16.94 15.419999999999998 0.0 12.81 
				21 9.25 4.49 0.0 6.0 
				22 13.08 6.479999999999999 8.620000000000001 8.83 
				23 19.07 0.0 0.0 14.95 
				24 14.54 8 7.4 10.26
			}\datatable

			\makeatletter
			\begin{tikzpicture}
			\begin{axis}[
			width=0.99\textwidth,
			y label style={at={(axis description cs:0.08,.5)}},
			ylabel=Cost of the solution,
			xtick=data,
			xticklabels={1,2,3,4,5,6,7,8,9,10,12,13,14,11,15,16,17,18,19,20,21,22,25,26},
			ymin=0,ymax=45,
			ybar stacked,
			bar width=4pt,
			legend style={
				font=\footnotesize,
				cells={anchor=west},
				legend columns=5,
				at={(0.3,-0.05)},
				anchor=north,
				xticklabel style = {font=\tiny,yshift=0.5ex}
			},
			]
			
			\addplot[fill=white] table[x index=0,y index=1] \datatable;
			\addplot[fill=lightgray] table[x index=0,y index=2] \datatable;
			\addplot[fill=black] table[x index=0,y index=3] \datatable;
			\legend{$k=1$,$k=2$,$k=3$}
			\end{axis}

			\begin{axis}[
			width=0.99\textwidth,
			xtick=data,
			xticklabels={1,2,3,4,5,6,7,8,9,10,12,13,14,11,15,16,17,18,19,20,21,22,25,26},
			ymin=0,ymax=45,
			legend style={
				font=\footnotesize,
				cells={anchor=west},
				legend columns=5,
				at={(0.81,-0.05)},
				anchor=north,
				xticklabel style = {font=\tiny,yshift=0.5ex}
			},
			]
			\addplot table[x index=0,y index=4] \datatable;
			\legend{k=0}
			\end{axis}
			\end{tikzpicture}
		\end{center}
				\caption{$k'=3$}
	\end{subfigure}
	\hfill
	\caption{Cost of the solutions for different instances}
	\label{fig:costRobust}
\end{figure}

Figure \ref{fig:costRobust} deals with the cost of designing failure-resilient networks; the number of the corresponding test instance is displayed on the $x$-axis. Each subfigure shows the cost of an optimal solution for the case where $k$ equals 0 (no arc deleted), 1, 2 and 3. The subfigure (a) corresponds to the case where $k'=0$ (no protection allowed), whereas (b), (c) and (d) correspond to the case where $k'$ is equal to 1, 2 and 3, respectively. This figure shows that designing a network resilient to even a small number of arc-failures can be costly (the cost increases greatly with the value of $k$). However, on subfigure (d), we can see that, by protecting a sufficiently large but still small subset of arcs on the test instances, one can obtain networks that are resilient to 1 or 2 arc deletions while maintaining a cost close to the optimal value of the case with no arc failures.\\

\section{Conclusion}

In this paper, we studied the design of robust networks, i.e., networks that are resilient to arc failures, where one wants to route a uniform flow from a source node to several sink nodes. More precisely, we first focused on the design of a robust arborescence that minimizes the losses if an arc failure occurs. We showed that a restriction of the problem is already hard, and then derived some new mathematical formulations to tackle the problem. Then, we considered the problem of designing a robust network that is able to route the flow, even after the deletion of any $k$ arcs. We derived three mathematical formulations for this problem. We finally tested all the methods that we proposed, and exhibited the associated computational results. The test instances were either randomly generated or obtained from wind power distribution networks, in which we aim to route the energy produced by some wind turbines to a substation that will then deliver this energy to the electrical grid. In a future work, we would like to improve the solving time of our models, in particular of the bilevel one, for instance by solving the associated subproblems using efficient heuristics whenever this is possible.


\bibliographystyle{plain}

\begin{thebibliography}{}

\end{thebibliography}


\begin{thebibliography}{99}
	\bibitem{baiou1997steiner}
	M. Ba{\"\i}ou, AR. Mahjoub, \textit{Steiner 2-edge connected subgraph polytopes on series-parallel graphs}, SIAM Journal on Discrete Mathematics, 10-3 (1997) 505-514.
	
	
	\bibitem{Bentz2016EdgeCapacitated}
	C. Bentz, MC. Costa, A. Hertz, \textit{On the edge capacitated Steiner tree problem}, CoRR, abs/1607.07082 (2016).
	
	\bibitem{bienstock2000strong}
	D. Bienstock, G. Muratore, \textit{Strong inequalities for capacitated survivable network design problems}, Mathematical Programming, 89-1 (2000) 127-147.
	
	\bibitem{biha2000steiner}
	MD. Biha, AR. Mahjoub, \textit{Steiner k-edge connected subgraph polyhedra}, Journal of Combinatorial Optimization, 4-1 (2000) 131-144.
		
	\bibitem{botton2013benders}
	Q. Botton, B. Fortz, L. Gouveia, M. Poss, \textit{Benders decomposition for the hop-constrained survivable network design problem}, INFORMS journal on computing, 25-1 (2013) 13-26.
	
	\bibitem{bousba1991finding}
	C. Bousba, L. Wolsey, \textit{Finding minimum cost directed trees with demands and capacities}, Annals of operations research, 33-4 (1991) 285-303.
	
	\bibitem{dahl1998cutting}
	G. Dahl, M. Stoer, \textit{A cutting plane algorithm for multicommodity survivable network design problems}, INFORMS Journal on Computing, 10-1 (1998) 1-11.
	
	
	\bibitem{du2013advances}
	DZ. Du, JM. Smith, JH. Rubinstein, \textit{Advances in Steiner trees} (Vol.6), Springer Science \& Business Media, (2013).
	
	\bibitem{garey1979computers}
   	MR. Garey, DS. Johnson, \textit{Computers and Intractability: A Guide to the Theory of NP-Completeness}, W. H. Freeman, (1979).
   	
	\bibitem{goemans1993survivable}
	M. Goemans, D. Bertsimas, \textit{Survivable networks, linear programming relaxations and the parsimonious property}, Mathematical Programming, 60-1-3 (1993) 145-166.
	
	\bibitem{goemans1993catalog}
	M. Goemans, YS. Myung, \textit{A catalog of Steiner tree formulations}, Networks, 23-1 (1993) 19-28.
	
	\bibitem{grotschel1995design}
	M. Grotschel, CL. Monma, M. Stoer, \textit{Design of survivable networks}, Handbooks in Operations Research and Management Science, 7 (1995) 617-672.
	
	\bibitem{hertz2012optimizing}
	A. Hertz, O. Marcotte, A. Mdimagh, M. Carreau, F. Welt, \textit{Optimizing the design of a wind farm collection network}, INFOR: Information Systems and Operational Research, 50-2 (2012) 95-104.
	
	\bibitem{hwang1992steiner}
	FK. Hwang, DS. Richards, P. Winter, \textit{The Steiner tree problem}, Elsevier, (1992).
	
	\bibitem{kerivin2002design}
	H. Kerivin, D. Nace, J. Geffard, \textit{Design of survivable networks with a single facility}, Universal Multiservice Networks, 2002. ECUMN 2002. 2nd European Conference on, IEEE (2002) 208-218.
	
	\bibitem{kerivin2005design}
	H. Kerivin, AR. Mahjoub, \textit{Design of survivable networks: A survey}, Networks, 46-1 (2005) 1-21.
	
	\bibitem{papadimitriou1978}
	CH. Papadimitriou, \textit{The complexity of the capacitated tree problem}, Networks, 8-3 (1978) 217-230.
	
	\bibitem{pillai2015offshore}
	AC. Pillai, J. Chick, L. Johanning, M. Khorasanchi, V. de Laleu, \textit{Offshore wind farm electrical cable layout optimization}, Engineering Optimization, (2015) 1-20.
	
	\bibitem{rajan2004directed}
	D. Rajan, A. Atamtürk, \textit{A directed cycle-based column-and-cut generation method for capacitated survivable network design}, Networks, 43-4 (2004) 201-211.
	
	\bibitem{ratliff1975finding}
	HD. Ratliff, GT. Sicilia, SH. Lubore, \textit{Finding the n most vital links in flow networks}, Management Science, 21-5 (1975) 531-539.
	
	\bibitem{stoer1994polyhedral}
	M. Stoer, G. Dahl, \textit{A polyhedral approach to multicommodity survivable network design}, Numerische Mathematik, 68-1 (1994) 149-167.
	

	
	
\end{thebibliography}

\end{document}